\documentclass[12pt]{amsart}

\usepackage{mathtools}
\usepackage[english]{babel}
\usepackage{cancel}
\usepackage[most]{tcolorbox}
\usepackage{float}
\usepackage[a4paper,top=35mm,bottom=30mm,left=30mm,right=30mm]{geometry}

\usepackage{latexsym,amsmath,amsthm,amssymb,amsfonts,mathrsfs,enumerate}
\usepackage{xcolor}
\usepackage[normalem]{ulem}
\usepackage{graphicx}
\usepackage[colorlinks=true, allcolors=blue]{hyperref}

\usepackage{tikz}
\usepackage{pgfplots}
\pgfplotsset{compat=1.18}

\newtcolorbox{greybox}{colframe=black,colback=gray!10,coltext=black,boxrule=0.5mm,arc=2mm,width=\linewidth,left=6pt,right=6pt,top=6pt,bottom=6pt,breakable,enhanced
}

\definecolor{caca}{rgb}{0.137,0.3,0.541}
\newtheorem*{thmA}{Theorem (I)}
\newcommand{\refThmA}{(I)\space}
\newcommand{\refThmAnospace}{(I)}
\newtheorem*{thmB}{Theorem (II)}
\newcommand{\refThmB}{(II)\space}

\newtheorem*{thmC}{Theorem (III)}

\newtheorem*{thmD'}{Theorem D'}

\newtheorem{thm}{Theorem}[section]

\newtheorem{prop}[thm]{Proposition}
\newtheorem{cor}[thm]{Corollary}

\newtheorem{lem}[thm]{Lemma}

\newtheorem{dfn}[thm]{Definition}
\newtheorem{rem}[thm]{Remark}
\newtheorem{remark}[thm]{Remark}

\newcommand{\diam}{\operatorname{diam}}
\newcommand{\R}{\mathbb{R}}

\newcommand{\ee}{\mathbf{e}}

\DeclareMathOperator{\width}{width}

\DeclareMathOperator{\Hess}{Hess}

\DeclareMathOperator{\dist}{dist} 
\DeclareMathOperator{\diverg}{div} 

\title[Collapsed translating solitons with planar wings]{Collapsed translating solitons\\with planar wings}
\author[E.S. Gama]{Eddygledson S. Gama}
\address[Gama]{
  Departamento de Matem\'atica,
   Centro de Ci\^encias Exatas e da Natureza, Universidade Federal de Pernambuco, Recife-PE, 50670-901, Brazil.
}
\email{eddygledson.gama@ufpe.br}

\author[F. Mart\'in]{Francisco Mart\'\i{}n}
\address[Mart\'in]{
  Departamento de Geometr\'\i{}a y Topolog\'\i{}a, 
  Instituto de Matem\'aticas IMAG,
  Universidad de Granada,
  18071 Granada, Spain.
}
\email{fmartin@ugr.es}

\author[N.M. M{\o}ller]{Niels M. M{\o}ller}
\address[M{\o}ller]{
  Copenhagen Centre for Geometry and Topology (GeoTop), Department of Mathematical Sciences,
  University of Copenhagen, 
  DK-2100 Copenhagen, Denmark.
}
\email{nmoller@math.ku.dk}



\begin{document}

\begin{abstract}
    We study complete embedded translating solitons in $\mathbb{R}^3$ with finite genus and finite entropy. We prove that such collapsed translators whose wings are all of planar type (in a weak topological sense) must be vertical flat planes.
    
    In particular, this implies the non-existence result for gluing and doubling constructions involving planes that was conjectured in \cite{HMW-SpruckVol} and announced in \cite{GMM-UniquenessTangentPlanes} as an application of the authors' techniques, under weaker assumptions than previous rigidity results by Khan and Lynch-Tinaglia in \cite{Ilyas,Lynch-Tinaglia,PacoCrelle}.
    
    A key ingredient is a sharp new a priori decay estimate in the exterior problem for the horizontal graphical translator equation, showing that bounded solutions to the nonlinear equation, and their gradients, are controlled by multiples of the Green's function of the linear drift Laplacian, regardless of the size of the given $L^\infty$ data. Combined with a graphical description of the infinite ends and the Neves–Tian flux identity, this yields the main rigidity result. 
    
    As another application, we obtain a full classification of collapsed embedded finite-genus translators of entropy $\lambda < 3$.
    
\end{abstract}

\maketitle
\tableofcontents

\section{Introduction}

A translating soliton (or \emph{translator}) for the mean curvature flow is a
surface $\Sigma^2\subset\mathbb{R}^3$ with $\partial\Sigma = \varnothing$ satisfying, for $\vec H$ the mean curvature vector of $\Sigma$,
\begin{equation}\label{TSE_geom}
        \vec H=\ee_3^{\perp}.
        \end{equation}

Equivalently, the one-parameter family $\Sigma_t=\Sigma+t\ee_3$ evolves by mean
curvature flow. Translators arise naturally as models for Type II singularity
formation and, at the same time, form a distinguished class of complete
weighted minimal surfaces. Their geometry displays a striking mixture of
rigidity and flexibility; see, for instance,
\cite{CSS07,MSHS,HIMW-2} and the references therein.

The global theory in $\mathbb{R}^3$ already contains several fundamental
classification results. Complete graphical translators include the bowl soliton,
the possibly tilted grim reaper cylinders, and the $\Delta$-wings. Building on the
convexity and asymptotic results of Wang and of Spruck--Xiao
\cite{Wang,Spruck-Xiao}, the complete translating graphs in $\mathbb{R}^3$
were classified by Hoffman--Ilmanen--Mart\'{\i}n--White
\cite{HIMW}. Rigidity under asymptotic hypotheses has been studied
from a number of viewpoints. In \cite{GamaMartinHyperplanes}, translators
$C^1$-asymptotic to pairs of half-hyperplanes were classified, while Khan
\cite{Ilyas} obtained a detailed description of the ends of translators with
finite total curvature. More recently, Lynch--Tinaglia
\cite{Lynch-Tinaglia} proved that a complete, properly immersed translator of
finite topology whose ends are $C^1$-asymptotic to vertical planes must itself
be a vertical plane.

The present paper is concerned with stronger classification results, via a different, and much weaker, asymptotic regime. We work with complete embedded \emph{collapsed} translators of finite genus and finite entropy. We remind the reader that complete translators of finite entropy are automatically proper in $\mathbb{R}^3$ \cite{Chini}. In the terminology used here, collapsedness of translators means that they are contained in slabs of finite width (two of the four cases in the classification of projections of convex hulls of translators in \cite{Chini-Moller}), while we do not require any sign on the mean curvature. It turns out that \emph{wings}, i.e. connected components of the surface minus large vertical cylinders, is the right geometric notion for the description of the behavior at infinity of such translators. This point of view was developed in the authors' previous work \cite[Section 4]{GMM-AJM}, where wings were introduced and studied, as mostly topological features, together with their relation to entropy. The concept of wings was subsequently sharpened from a partial differential equations perspective in
\cite[Theorem A]{GMM-UniquenessTangentPlanes}, where uniqueness of limiting planes at
infinite times and a decomposition of collapsed translators into planar and grim
reaper type regions were established. Other rigidity and non-existence results
for collapsed translators, obtained by weighted potential-theoretic methods,
can be found in \cite{IRM}; low-entropy translators in slabs were studied earlier by Chini \cite{ChiniThesis,Chini} and the authors \cite{GMM-AJM}.

A natural question suggested by this structure theory, as pointed out in \cite{HMW-SpruckVol}, is whether a non-flat collapsed translator can have only planar wings. Our first main result gives a negative answer, as already announced in \cite{GMM-UniquenessTangentPlanes}. See Figure \ref{fig:planarwings}.

\begin{thmA}
Let $\Sigma^2\subseteq\mathbb{R}^3$ be a complete, embedded, connected, collapsed translating soliton with finite genus and finite entropy. If every
wing of $\Sigma$ is of planar type, then $\Sigma$ is a vertical plane.
\end{thmA}

\begin{remark}\label{rem:planar_wings}
We recall from \cite[Definition 7.3]{GMM-AJM} (see also \cite{GMM-UniquenessTangentPlanes}) that wings of planar type for a collapsed translator are defined as those connected components of $\Sigma\setminus \mathrm{Cyl}_R(0)$, $R \gg 1$ and $\mathrm{Cyl}_R(0)=\{(x_1,x_2,x_3): x_1^2+x_2^2\leq R^2\}$, for which the $x_3$-coordinate is unbounded in both directions on all curves obtained from slicing perpendicularly to the slab containing $\Sigma$. See Figure \ref{fig:planarwings}, where the slice curves seen in the plane of the paper have this property.

 In particular, translators whose ends are bounded graphs over non-compact vertical annuli (see Definition \ref{dfn:collapsed_graphical_end}) are trivially examples of surfaces whose wings are all of planar type.
\end{remark}

Theorem \refThmA is proved below as Theorem~\ref{thm:planar_wings}. The point is that
``planar wing'' is substantially weaker than assuming from the outset that an end is strongly asymptotic to a vertical plane, which would require a priori knowing that the solution is a graph that $C^1$-converges to a constant. In particular, no decay rates, finite total curvature, or $C^1$-asymptotic expansions are demanded by our hypotheses. Thus Theorem \refThmA holds significantly more generally than the results in \cite{Lynch-Tinaglia} (see Remark \ref{rem:Neves-Tian}): Lynch--Tinaglia's rigidity theorem starts from strong $C^1$-asymptotic assumptions on each end, whereas in the present context, the required limits and decay are produced as part of the argument from just the wing structure that is available under very weak assumptions, that are essentially topological (as in Remark \ref{rem:planar_wings}). Namely wings are well-defined for the whole class of finite-genus, finite-entropy collapsed translators (see Figure \ref{fig:planarwings}). Together with Khan's
work in \cite{Ilyas}, this also gives a separate route to the rigidity phenomenon for embedded translators in the special case of finite total curvature.

\begin{figure}[h!]
\centering
\setlength{\fboxrule}{0pt}
\fboxsep=5pt
\fbox{

\includegraphics[width=0.22\linewidth]{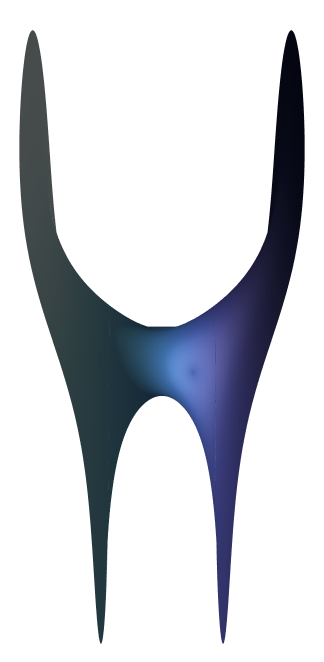}
\includegraphics[width=0.3\linewidth]{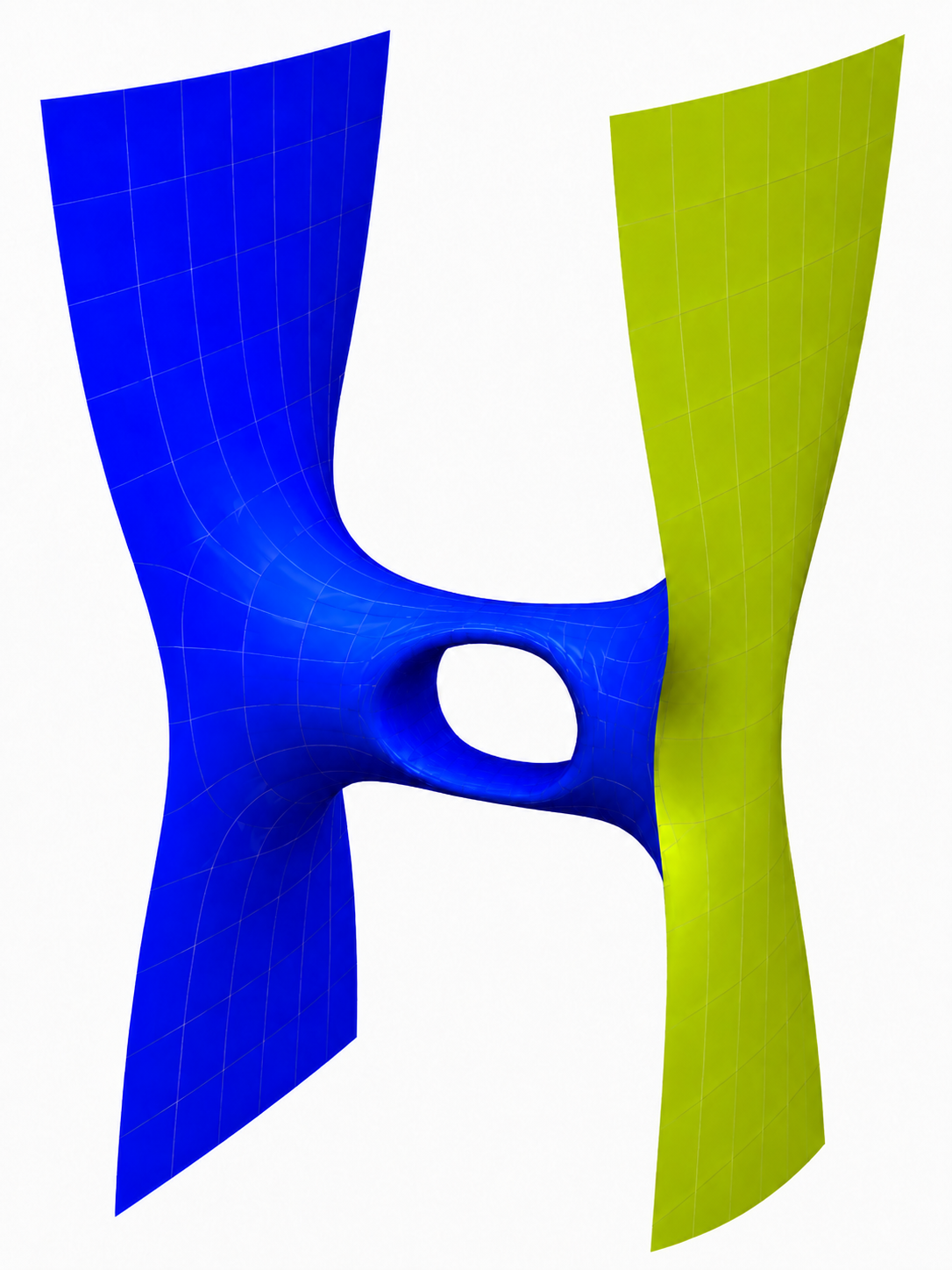}}
\caption{Schematic configurations of some hypothetical translating solitons which would have only \emph{wings of planar type} understood in the weak differential-topological sense as in Remark \ref{rem:planar_wings}, yet reflect the anisotropy of the equation \eqref{TSE_geom} by not being asymptotic to the same planes in all directions, i.e. they are not assumed \emph{a priori} to satisfy any strongly planar asymptotic conditions in the sense of \cite{Ilyas,Lynch-Tinaglia}. Theorem~\ref{thm:planar_wings} shows that just assuming a given connected collapsed translating soliton of finite genus and finite entropy only has planar wings, then it must in fact be a vertical flat plane. This non-existence result was conjectured in 2023 in \cite{HMW-SpruckVol} and announced in \cite{GMM-UniquenessTangentPlanes}.}
\label{fig:planarwings}
\end{figure}

As graphicality of the ends will follow from the geometric methods in \cite{GMM-AJM,GMM-UniquenessTangentPlanes}, the main new analytic input consists of new results on an exterior problem for the horizontal graphical translator equation. Suppose therefore that
\begin{equation}\label{eq:horiz_graph}
X(x_2,x_3)=\bigl(u(x_2,x_3),x_2,x_3\bigr),
\qquad
u:A_{R_0}:=\mathbb{R}^2\setminus B_{R_0}(0)\longrightarrow\mathbb{R},
\end{equation}
where $u$ is smooth and globally bounded, the boundedness being motivated geometrically from the assumed collapsedness. 
Here we have made use of the translation invariance of the translator equation $\eqref{TSE_geom}$ and moved the point of interest to the origin in $\mathbb{R}^3$. Writing $W=\sqrt{1+|\nabla u|^2}$, the translator
equation is
\begin{equation} \label{TSE_divform}
\operatorname{div}\left(\frac{\nabla u}{W}\right)
=-\frac{\partial_3u}{W},
\end{equation}
or equivalently the weight can be brought inside the divergence,
\begin{equation}
\operatorname{div}\left(e^{x_3}\frac{\nabla u}{W}\right)=0.
\label{eq:weighted-translator-intro}
\end{equation}
The equation's linearization at a vertical plane is the drift Laplacian
$$
        L=\Delta+\ee_3\cdot\nabla,
$$
whose Green's function on $\mathbb{R}^2$ is
$$
G_L(x)=\frac{1}{2\pi}e^{-x_3/2}K_0\left(\frac{|x|}{2}\right),
$$
where $K_0$ is the modified Bessel function of the second kind. We prove that
this linear Green's function gives the sharp pointwise scale for every bounded
nonlinear exterior solution, with no smallness assumption on the $L^\infty$ data.

\begin{thmB}
Let $u$ be a smooth bounded solution of the horizontal graphical translator
equation on $\mathbb{R}^2\setminus\Omega$, where $0\in\Omega$ is a bounded set. Then
there exists $u_\infty\in\mathbb{R}$:
\begin{equation}
        |u(x)-u_\infty|\le C\,G_L(x), \quad \quad\dist(x,\Omega) > 1,
\end{equation}
\begin{equation}
        |\nabla u(x)|\le C\,G_L(x), \quad \quad\dist(x,\Omega) > 1.
\end{equation}
In particular, the natural sharp asymptotical expansion holds true:
\begin{equation}
u(x)=u_\infty+
O\!\left(|x|^{-1/2}
\exp\!\left[-\frac{x_3+|x|}{2}\right]\right)
\quad\text{as}\quad|x|\to\infty.
\end{equation}

\end{thmB}

This is Theorem~\ref{thm:optimal_decay} below. It strengthens the asymptotic
information available in our previous work
\cite{GMM-UniquenessTangentPlanes} (see also Remark \ref{rem:ray_counterex}), where directional limits along rays and derivative
decay properties were obtained for general collapsed translators, but where uniform sharp
exterior estimates were not available, as such do not hold true in the generality treated there. In the finite-total-curvature setting,
Khan \cite{Ilyas} obtained a weaker quantitative decay estimate (see Remark \ref{rem:compare_Ilyas}); the bound in
Theorem \refThmB has the optimal exponential factor dictated by the Green's function
of $L$. The philosophy that nonlinear asymptotics should be governed by the
linearized Green's function has close analogues in other geometric elliptic
problems; compare, for example, the refined asymptotic analysis of
Korevaar--Mazzeo--Pacard--Schoen \cite{KMPS} for the singular Yamabe
equation.

In particular, Theorem (II) shows the remarkable phenomenon that the manifestly anisotropic equation \eqref{TSE_geom}, or equivalently \eqref{TSE}, only admits $L^\infty$ exterior domain solutions whose limits are in fact isotropic at infinity. See Figure \ref{fig:planarwings}.

The proof of Theorem \refThmB starts from well-known methods, but to obtain the optimal result presented here requires carefully handling of the anisotropy and the finer aspects of its interaction with the nonlinearities of the equation, leading us to develop a rather delicate technique for anisotropic decay bootstrapping that is well-suited to this situation. We first construct a one-parameter family of explicit $L$-harmonic barriers, obtained from linearized tilted grim-reaper-type supersolutions for the nonlinear equation, and minimize over the family to obtain a global anisotropic exponential bound. We then combine this estimate with the a priori Hessian estimates from
\cite{GMM-UniquenessTangentPlanes} and an anisotropic $L^\infty$ interpolation argument
to gain a priori control of first the gradient and then the nonlinear error. A final comparison argument for the Yukawa conjugate of $L$ upgrades the preliminary estimate to the sharp Green's function bound. An important feature is that the argument works for bounded nonlinear solutions of arbitrary supremum norm.

Once Theorem \refThmB has been proven, Theorem \refThmA follows readily. Namely, the structure theory from
\cite{GMM-UniquenessTangentPlanes}, together with a topological covering
argument, shows that sufficiently far out every set of left and right pair of planar wings gives rise to a single-valued
horizontal graph over the complement of a vertical disk. Theorem~\refThmB then supplies enough
decay to show that the vertical flux through the boundary circles of those
graphical ends tends to zero. On the other hand, for every compact translator
$M$ with smooth boundary one has the identity
\begin{equation}\label{NTas}
        \int_M |\vec H|^2\,dA
        =
        \int_{\partial M}\langle\eta,\ee_3\rangle\,d\sigma,
\end{equation}
where $\eta$ is the outward unit conormal. This elementary translator flux identity is closely related to the identities used by Neves--Tian
\cite{Neves-Tian} and is also the key mechanism in \cite{Lynch-Tinaglia} (see also \eqref{PohoNTas} below). Applying it to larger and larger compact
truncations of $\Sigma$ and using the decay of the boundary flux gives

$$
        \int_\Sigma |\vec H|^2\,dA=0.
$$

\noindent{}Thus $\vec H\equiv0$, and the translator equation \eqref{TSE_geom} forces $\Sigma$ to equal a vertical plane. This argument can also be interpreted as a Pohozaev-style non-existence result, see Section \ref{sec:Poho}.

Theorem \refThmA has an important further consequence: we are able to classify the collapsed translators of finite genus and entropy $\lambda < 3.$ In fact, we prove the following:
\begin{thmC}
   Let $\Sigma$ be a complete, connected, embedded, collapsed translator in $\R^3$ with finite genus and entropy $\lambda(\Sigma) < 3.$ Then, $\Sigma$ is either a vertical plane, a possibly tilted grim-reaper cylinder, or a $\Delta$-wing.
\end{thmC}
This theorem, which we prove below as Corollary \ref{cor:entropy-at-most-two}, uses a previous characterization result of grim reapers and $\Delta$-wings obtained by the authors in \cite[Theorem F]{GMM-UniquenessTangentPlanes}.

Besides proving Theorem \refThmAnospace, the sharp exterior estimate has consequences that
are useful independently of the global rigidity theorem. In particular, a
planar graphical end carries a conserved weighted horizontal flux, and this
flux detects an obstruction to improving the decay estimates from $O(G_L)$ to $o(G_L)$, and we show in Proposition \ref{prop:flux-obstruction-rewritten} that the only collapsed graphical ends with compact boundary in a vertical plane and $o(G_L)$ decay are subsets of vertical planes in $\mathbb{R}^3$. Already for the linearized equation, different angular Fourier modes occur at
the same Green's function order, so one should not expect a general bounded
exterior solution to have a leading term which is merely a constant multiple
of $G_L$. We develop these flux invariants in
Section~\ref{sec:horizontal-flux}. They also enter naturally in the study of
free-boundary
rigidity discussed later in the paper.

The latter discussion is related to a broader circle of Liouville and boundary
rigidity questions. Compact translators with planar boundary were studied by
Pyo \cite{Pyo}, who obtained symmetry and balancing results under geometric
hypotheses on boundary curves lying in planes orthogonal to the translating
direction. For other geometric elliptic equations, half-space and
free-boundary Liouville theorems have recently been proved for capillary
minimal graphs by Wang--Wei--Zhang \cite{Capillary} and for anisotropic
minimal graphs by Wang--Wei--Xia--Zhang \cite{Anisotropic}. In the translator
setting, and more generally for $[\varphi,\ee_3]$-minimal surfaces
\cite{Trivino}, the positivity of the vertical flux identity gives a
particularly direct rigidity mechanism.

The paper is organized as follows. In Section~\ref{sec:exterior-problem} we
prove the sharp exterior decay theorem. We then establish the graphical
decomposition of planar wings and prove Theorem~\ref{thm:planar_wings}. Section \ref{sec:entropy-two} is devoted to the classification of collapsed translators with entropy strictly less than three. In
Section~\ref{sec:horizontal-flux} we study the weighted horizontal fluxes of
planar ends and their relation to the optimal decay rate. The remaining
sections investigate the associated
boundary-flux formulas (Section \ref{sec:vert_flux}), together with their rigidity consequences, and the link to Pohozaev-type identities (Section \ref{sec:Poho}).\\

\noindent{}\textbf{Acknowledgments.} E.S. Gama was partially supported by CNPq/Brazil Grant: 313412/2025-1. E.S. Gama and F. Mart\'in were partially supported by the NSF grant DMS-1928930, by the grant PID2024-156031NB-I00 funded by  MICIU / AEI / 10.13039/ 501100011033. N.M. M\o{}ller was partially supported by DFF Sapere Aude 7027-00110B, by Carlsberg Foundation Semper Ardens CF21-0680 and by CPH-GEOTOP-DNRF151.

Apart from drafting this AI use statement, generative AI was used only to verify small calculus facts, to correct typos and to search for references. All mathematical content is the authors' own original work.

\section{Sharp decay in the nonlinear exterior problem} \label{sec:exterior-problem}

For nonlinear problems it is very natural to expect, but generally difficult to show, that well-behaved classes of solutions are bounded by the Green's function of the corresponding linearized equation. See e.g. Korevaar--Mazzeo--Pacard--Schoen \cite{KMPS}, for the singular Yamabe problem, a semilinear equation whose rotational symmetry was brought into play via mode analysis. We remind the reader that equations \eqref{TSE_geom} (see also \eqref{TSE}) are not small perturbations at infinity of the classical Euclidean minimal surface equation, as the drift term continues to influence the solutions far out, leading to the global analysis differing markedly from the classical case.

Here, despite the anisotropy (see Figure \ref{fig:planarwings}), we are able to prove such a property, as announced in \cite[Remark 4.22]{GMM-UniquenessTangentPlanes}, for exterior domains $\mathbb{R}^2\setminus\Omega$, with $\Omega\ni 0$ bounded:
\begin{thm}[Optimal decay for $L^\infty$ solutions of arbitrary norm]\label{thm:optimal_decay}
Suppose that the function $u: \mathbb{R}^2\setminus \Omega \to \mathbb{R}$ is a smooth solution with $\sup_{\mathbb{R}^2\setminus \Omega} |u| < \infty$ to
\begin{equation}\label{TSE}
    \diverg\left(\frac{\nabla u}{\sqrt{1+|\nabla u|^2}}\right)=-\frac{\partial _{3} u}{\sqrt{1+|\nabla u|^2}},
\end{equation}
which is equivalent to, writing $L:= \Delta + \ee_3\cdot\nabla$, the equation
\begin{equation}\label{TSE_split}
Lu = P(\nabla u, \operatorname{Hess} \,u), \quad \mathrm{for}\quad P(\nabla u, \operatorname{Hess}\, u):=\frac{\operatorname{Hess}  u(\nabla u, \nabla u)}{1+|\nabla u|^2}.
\end{equation}

Then there exists $u_\infty\in\mathbb{R}$, and $C_0 = C_0(\diam(\Omega),\|u\|_{L^\infty})$, such that 
\begin{equation}\label{eq:final_G_bound}
|u(x) - u_\infty|\leq C_0\,G_L(x), \quad \quad\dist(x,\Omega) > 1,
\end{equation}
\begin{equation}\label{eq:G_bound_grad}
|\nabla u(x)|\leq C_0G_L(x),  \quad \quad\dist(x,\Omega) > 1,
\end{equation}
where $G_L(x)=\frac{1}{2\pi}e^{-\frac{x_3}{2}}K_0\left(\tfrac{|x|}{2}\right)$ is the Green's function for $L$, with $K_0$ being the modified Bessel function of the second kind.

In particular:
\begin{equation}\label{asympt_u_spellout}
u = u_\infty + O\left(|x|^{-1/2}\exp\left(-\tfrac{x_3 + |x|}{2}\right)\right)\quad\text{as}\quad|x|\to\infty.
\end{equation}
\end{thm}

\begin{rem}\label{rem:compare_Ilyas}
 The new a priori decay estimates in \eqref{eq:final_G_bound}--\eqref{asympt_u_spellout} give a more than squared improvement over the rate of order $O\big(|x|^{-\tfrac{1}{4} + \varepsilon}\big)\exp\left(-\tfrac{x_3 + |x|}{4}\right)$ for $\varepsilon > 0$ found by Khan in \cite{Ilyas}, where (starting from the assumption $\int_\Sigma |A|^2 < \infty$) the additional condition $|\nabla u| = o(1)$ as $|x|\to\infty$ was assumed.
\end{rem}

\begin{rem}
 The estimate \eqref{eq:G_bound_grad}, valid for arbitrarily large data, is in all directions away from $-\ee_3$ an exponential improvement over the general gradient decay obtained in \cite{GMM-UniquenessTangentPlanes} (here as \eqref{eq:grad_up} below). Hence, it also improves on the weaker isotropic small data gradient estimate in Lynch--Tinaglia's paper \cite{Lynch-Tinaglia}, which in turn follows directly by usual isotropic $L^\infty$ interpolation from the ``softer'' of the Hessian estimates in \cite{GMM-UniquenessTangentPlanes} (see \eqref{est:general} in Proposition \ref{prop:poelsekapsel} below), of order $O(|x|^{-1/2})$ times the supremum norm over circles. This yielded only gradient decay of order $O(|x|^{-\tfrac{3}{4}+\varepsilon})$, when combined with Khan's decay estimate for the special case of finite total curvature. See also Remark \ref{rem:sharpness} below.
\end{rem}

To prove Theorem \ref{thm:optimal_decay}, we will therefore first do a careful analysis of the anisotropic problem, in order to get strong enough decay estimates.

\begin{lem}\label{lem:first_decay}
Suppose that $u: \mathbb{R}^2\setminus \Omega \to \mathbb{R}$ is a smooth exterior $L^\infty$ solution to the nonlinear equation \eqref{TSE}, where $0\in\Omega$ is a smooth bounded domain. Then:
\begin{equation}\label{eq:smeared_sup_bound}
\exists u_\infty:\: 
|u(x) - u_\infty|\leq C(\diam(\Omega))\|u\|_{L^\infty}\exp\left[-\tfrac{x_3 + \sqrt{x_2^2 + x_3^2}}{2}\right], \quad 
\dist(x,\Omega) > 1.
\end{equation}
\end{lem}

For the proof of Lemma \ref{lem:first_decay}, it appears difficult to find sharp nonlinear barriers directly for the problem, i.e. from modifications of the Green's function itself, as scaling invariance with respect to the $L^\infty$ data is not generally guaranteed, unless the barrier is for example both $L$-harmonic and gradient convex, which decouples the terms with different scaling behavior (see \eqref{eq:supsol} below). The latter property, gradient convexity, meaning $\mathrm{Hess}\,w(\nabla w, \nabla w)\geq 0$, however fails for the Green's function itself. There is also the subtle point that even if one finds supersolutions to the nonlinear problem, such barriers might not necessarily work at all scales, i.e. for arbitrary values of the data $R > 0$ and $\|u\|_{\infty}$. See also Remark \ref{rem:not_barrier} below.

We will therefore instead make use of a whole one-parameter family of simpler barriers with good properties, coming from Killing fields from families of tilted grim reaper cylinders, which are rotated, scaled and translated appropriately around in space. Each of these individually satisfies both $L$-harmonicity and gradient convexity, allowing us to get a family of estimates on $u$ valid for arbitrarily large norm solutions, at which time we can then minimize the resulting inequalities over this family.

An analytic starting point for the argument is the idea (similarly to in \cite{GMM-UniquenessTangentPlanes}, see also \cite{KKM, NielsThesis}) that once sufficiently strong decay estimates have been obtained for the second derivatives, then the drift term $\ee_3\cdot\nabla$ can be split off from the linear operator and explicitly integrated as an ODE, in order to prove that limits exist, together with estimates. See \eqref{bdry_limit} below.

Before beginning the detailed proofs, we therefore recall the authors' sharpened versions of the Evans-Spruck and Ecker-Huisken a priori gradient and curvature estimates from \cite{GMM-UniquenessTangentPlanes}:

\begin{prop}{\cite[Prop. 3.9 \& 3.10]{GMM-UniquenessTangentPlanes}}\label{prop:poelsekapsel}
Suppose that the noncompact domain $D\subseteq\mathbb{R}^2$ has the upward capsules property, that is, for every $p\in D$ holds
\begin{equation}\label{poelser}
S^+_\rho : = \bigcup_{t\in [0,\rho^2)} \left(B_\rho(p) + t\ee_{3}\right) \;\subseteq\; D,\quad\mathrm{for}\quad \rho = \dist(p,\partial D).
\end{equation}

Then for $u:D\to\R$ any smooth solution of the translating soliton equation for sideways graphs \eqref{TSE} with $\|u\|_\infty = \sup_D |u| < \infty$, the following a priori gradient and curvature estimates hold, for some constant $K=K(\|u\|_\infty)$, for any $ p\in D$ with $\dist(p,\partial D) > 1$:
 \begin{equation}\label{eq:grad_up}
|\nabla u|(p)\leq K\bigl[\dist(p,\partial D)\bigr]^{-1},
 \end{equation}
\begin{equation}\label{Hess_up_improved}
|\Hess u|(p)\leq K\bigl[\dist(p,\partial D)\bigr]^{-2}.
 \end{equation}

 For general domains, where the capsule property may fail, weaker bounds hold:
\begin{equation}\label{est:general}
|\nabla u|(p)\leq K\bigl[\dist(p,\partial D)\bigr]^{-1/2},\quad
|\Hess u|(p)\leq K\bigl[\dist(p,\partial D)\bigr]^{-1}.
 \end{equation}
\end{prop}

We remind the reader that a key tool in the authors' \cite{GMM-UniquenessTangentPlanes} was the following sharp elliptic-parabolic estimate for the drift Laplacian:

\begin{lem}{\cite[Lemma 3.7]{GMM-UniquenessTangentPlanes}}\label{lem:varme_gradient_inhomogent}
There exists a numerical constant $K$ such that for any $u:S^+_\rho(p)\to \mathbb{R}$, where $S^+_\rho$ is an upward capsule as in \eqref{poelser} of arbitrary width $\rho > 0$, the following gradient estimate holds for $L=\Delta + \ee_3\cdot \nabla$ the drift Laplacian:
\begin{equation}\label{drift_poelse}
 |\nabla u|(p)\leq K\left(\frac{1}{\rho}\left \|u\right\|_{L^\infty(S^+_\rho(p))} + \rho\left\|Lu\right\|_{L^\infty(S^+_\rho(p))}\right).
\end{equation}
\end{lem}

With these a priori estimates for \eqref{TSE} and its linearization in place, we can now begin the proof of the first exponential decay lemma.

\begin{proof}[Proof of Lemma \ref {lem:first_decay}]\label{def:barriers}
Take $R_0 > 0$ sufficiently large that $\overline{\Omega}\subseteq B_{R_0-1}(0),$ and thus $\dist(\partial\Omega, \partial B_{R_0}(0))>1$. Given $(x_2,x_3)\in \R^2\setminus B_{R_0}(0)$, and $a\in\mathbb{R}$, we define:
\begin{equation}\label{b_familien}
b_{a,R_0}(x_2, x_3) = \exp\left( - \tfrac{x_3 + ax_2 - R_0\sqrt{1+a^2}}{1+a^2} \right).
\end{equation}
These are appropriate rescalings, depending on the two real parameters $(a,R_0)$, of the tilted grim reaper-derived barriers already used in the proof of \cite[Theorem 4.4]{GMM-UniquenessTangentPlanes}. We now recap and sharpen their basic properties as barriers. 

Firstly, the functions $b_{a,R_0}$ are $L$-harmonic, meaning that $
Lb_{a,R_0} = 0$, as a short calculation shows. Using $|\nabla b_{a,R_0}|^2 = b_{a,R_0}^2/(1+a^2)$ and $\Hess b_{a,R_0}
=
\frac{b_{a,R_0}}{(1+a^2)^2}
\left(\begin{smallmatrix}
a^2 & a \\
a\phantom{2} & 1
\end{smallmatrix}\right)$, these functions also satisfy strict gradient convexity (we note that in \cite{GMM-UniquenessTangentPlanes}, a softer convexity argument was used to argue nonnegativity, for superharmonicity in the nonlinear equation):
\begin{equation}\label{gradient_convexity}
\Hess b_{a,R_0}(\nabla b_{a,R_0},\nabla b_{a,R_0}) = \frac{b_{a,R_0}^3}{(1+ a^2)^2} > 0.
\end{equation}

Therefore, for any scale $\sigma \in \mathbb{R}_+$, and shift $\delta > 0$, setting $v=\sigma b_{a,R_0}+\delta$, we have (as in the argument for wedges in \cite{GMM-UniquenessTangentPlanes}),

\begin{equation}
\label{eq:supsol}
 Lv-\frac{\Hess v(\nabla v,\nabla v)}{1+|\nabla v|^2}
 =-\frac{\sigma^3\Hess b_{a,R_0}
       (\nabla b_{a,R_0},\nabla b_{a,R_0})}
       {1+\sigma^2|\nabla b_{a,R_0}|^2}<0,
\end{equation}
owing to the good signs of the separate terms, so that arbitrarily scaled and shifted $\sigma b_{a,R_0} + \delta$ are still supersolutions to the nonlinear equation $Lu - P(\nabla u,\Hess u) = 0$, i.e. to \eqref{TSE}.

Note now that each $b_{a,R_0}\geq 1$ on $B_{R_0}(0)$, as by Cauchy-Schwarz:
\begin{equation}\label{eq:b_lower_bound}
\forall (x_2, x_3) \in B_{R_0}:\quad
 x_3 + ax_2 \leq 
 \sqrt{x_2^2+x_3^2}\sqrt{1+a^2}
 \leq
 R_0\sqrt{1+a^2}.
\end{equation}

We denote the slanted upper half-plane given by all points where $b_{a,R_0}\leq 1$ by $H_{a,R_0}$,
\begin{equation}\label{def:H-plane}
H_{a,R_0}:=\left\{x_3 + ax_2 - R_0\sqrt{1+a^2} \geq 0\right\},    
\end{equation}
and note that its boundary line $\partial H_{a,R_0}$ is tangent to $\partial B_{R_0}$ at the point
\begin{equation}\label{p_point}
p_{a,R_0} = \left(\tfrac{aR_0}{\sqrt{1+a^2}},\tfrac{R_0}{\sqrt{1+a^2}}\right).
\end{equation}

We will be minimizing over the whole family of barriers $\{b_{a,R_0}\}_{a\in\mathbb{R}}$, asserting the following pointwise upper bound, where we denote $|x| = \sqrt{x_2^2 + x_3^2}$:\\

\noindent{\underline{Claim:}}
\begin{equation}
    \label{eq:exp_decay_claim}
\inf_{a \in \mathbb{R}} b_{a,R_0}(x_2, x_3) \leq C(R_0)\exp\left(-\tfrac{x_3 + \sqrt{x_2^2 + x_3^2}}{2}\right), \quad \text{for}\quad x_3\neq -|x|.
\end{equation}

To prove the claim in \eqref{eq:exp_decay_claim}, let $\ell_a = -\log(b_{a,R_0})=(x_3 + ax_2 - R_0\sqrt{1+a^2})/(1+a^2)$, whose supremum over $a\in\mathbb{R}$ we will then need to lower bound. For any given point $(x_2^0,x_3^0)$ not on the negative $x_3$-axis, i.e. denoting $r_0=\sqrt{(x_2^0)^2 + (x_3^0)^2}$ we assume $x_3\neq -r_0$, then pick and fix $a_0 = \frac{x_2^0}{x_3^0+r_0}$. Now we calculate that $(x_2^0)^2 = (r_0 - x^0_3)(r_0 + x^0_3)$, that $1+a_0^2 = 2r_0/(x_3^0+r_0)$ and $x_3^0+a_0x_2^0 = r_0$. Thus, we see:

\[
\ell_{a_0}(x_2^0,x_3^0) = \tfrac{x_3^0+r_0}{2} - R_0\sqrt{\tfrac{x_3^0+r_0}{2r_0}}\geq \tfrac{x_3^0+r_0}{2} - R_0,
\]
so that also
\[
\sup_{a\in\mathbb{R}} \ell_{a}(x_2^0,x_3^0) \geq \frac{x_3^0+r_0}{2} - R_0,
\]
proving the claim \eqref{eq:exp_decay_claim} by defining the explicit constant $C(R_0):= \exp(R_0)$.

Considering now any $L^\infty$ exterior solution $u$ to \eqref{TSE},
we need the following uniform convergence result for $u$ on arcs in slanted half-spaces, recalling the definition of $p_{a,R_0}$ in \eqref{p_point}:
\begin{equation}\label{bdry_limit}
\exists \; u_\infty\in\mathbb{R}\:\, \forall a\in\mathbb{R}:\quad 
\lim_{R\to\infty}\left[\sup_{\{H_{a,R_0}\,\cap\,\partial B_R(p_{a,R_0})\}} |u - u_\infty|\right] = 0,
\end{equation}
where we recall \eqref{def:H-plane}. The convergence in \eqref{bdry_limit} is now proven using the techniques from \cite{GMM-UniquenessTangentPlanes}. Namely, using first that $\dist(\partial\Omega,\partial B_{R_0}(0)) > 1$, we may enlarge $H_{a,R_0}$ to an open domain $W\supsetneq H_{a,R_0}$ whose boundary is given by the two half-lines of slopes $a_-,a_+$, with $a_- < a < a_+$, which meet at the point $p_{a,R_0}$ in \eqref{p_point}, and such that $W\subset\mathbb{R}^2\setminus\Omega$. This domain $W$ is also of upward capsule type, so that the estimates above in Proposition \ref{prop:poelsekapsel} still hold for $W$. Notice now that if we denote
\[
d_0 := \dist\left(\partial W, H_{a,R_0}\,\cap\,\partial B_1(p_{a,R_0})\right) > 0,
\]
then by scaling invariance of $\partial W$ around $p_{a,R_0}$, we have that
\[
\dist\left(\partial W, H_{a,R_0}\,\cap\,\partial B_R(p_{a,R_0})\right) = Rd_0.
\]

Thus, for any $p\in H_{a,R_0}$, we note that $p\in H_{a,R_0}\,\cap\,\partial B_R(p_{a,R_0})$ for $R = |p - p_{a,R_0}|$, so that $\dist(\partial W,p) \geq d_0|p - p_{a,R_0}|$. Using thus $Lu = P(\nabla u, \Hess u)$ with this in Proposition \ref{prop:poelsekapsel} gives $|\partial_3 u| \leq |\Delta u| + |P(\nabla u, \Hess u)| \leq C|p - p_{a,R_0}|^{-2}$. But then the limits $u_\infty := \lim_{s\nearrow\infty}u(p_0 + s\ee_3)$ exist for all $p_0$, with a value that, by the gradient estimate for upward capsule domains in \eqref{eq:grad_up}, is furthermore independent of the slant $a$. Thus, we conclude the basic representation formula $u(p_0) = u_\infty - \int_0^\infty\partial_3u(p_0 + s\ee_3)ds$ (similarly to in \cite{GMM-UniquenessTangentPlanes,KKM, NielsThesis}), which by explicit integration finally gives an estimate uniform in $p_0\in H_{a,R_0}\,\cap\,\partial B_R(p_{a,R_0})$, proving the limits claimed in \eqref{bdry_limit}:
\[
|u(p_0) - u_\infty| \leq C\frac{\sqrt{1+a^2}
\left(\frac{\pi}{2}+\arctan|a|\right)}{R}\to 0,\quad\mathrm{as}\quad R\to +\infty.
\]

We will now be using the maximum principle for \eqref{TSE}, e.g. \cite[Theorem 1]{Serrin70}, on the compact exhaustions of the exterior of the disk, $B_R(p_{a,R_0})\setminus B_{R_0}(0)$, for $2R_0 < R\nearrow \infty$. This will work owing to \eqref{bdry_limit}, where we remind the reader that without having shown existence of a limit at infinity, the maximum principle is known to fail in general, for this equation, in the noncompact case (see the discussions in \cite{GMM-UniquenessTangentPlanes}). We write the argument out in detail, which we note is similar in spirit to the one given for wedges contained in \cite{GMM-UniquenessTangentPlanes}.

Fix $a\in \mathbb{R}$ and $R > 2R_0$, as well as $\delta >0$ arbitrary. As an upper barrier, we can now use $w = 2\|u\|_{L^\infty}b_{a,R_0} + u_\infty + \delta$, which is superharmonic by the scale invariance of the nonlinear barrier condition for this special class, as shown in \eqref{eq:supsol}. Seeing as $b_{a,R_0} \geq 1$ on $H_{a,R_0}^c$, and thus on $\partial B_{R_0}$, due to $\overline{B_{R_0}}\subset H_{a,R_0}^c$, see \eqref{eq:b_lower_bound}, we have $u\leq w$ on the inner boundary $\partial B_{R_0}$. The outer boundary $\partial B_{R}(p_{a,R_0})$ splits into two pieces by intersection with $H_{a,R_0}$ and its complement. On the piece $\partial B_{R}(p_{a,R_0})\cap H_{a,R_0}^c$, that $u\leq w$ follows again from $b_{a,R_0} \geq 1$ there, and for $\partial B_{R}(p_{a,R_0})\cap H_{a,R_0}$ the limit \eqref{bdry_limit} means that for $R$ large enough (depending on $u$ and $\delta$), it also holds that $u\leq w$ on $\partial B_{R}(p_{a,R_0})\cap H_{a,R_0}$. Thus, by the maximum principle, $u\leq w$ on $B_R(p_{a,R_0})\setminus B_{R_0}(0)$, for all $R$ large. Tending $R\nearrow\infty$, we conclude that $u\leq w$ on $\mathbb{R}^2\setminus B_{R_0}(0)$.

Letting finally $\delta \searrow 0$, we have thus shown that, for each fixed $a\in\mathbb{R}$:
\begin{equation}\label{eq:u_bound_by_b}
|u - u_\infty| \leq 2\|u\|_{L^\infty}b_{a,R_0}, \quad\text{on}\quad \mathbb{R}^2\setminus{B_{R_0}}.
\end{equation}

Invoking now the property in \eqref{eq:exp_decay_claim}, taking the pointwise infimum over all $a\in\mathbb{R}$ in \eqref{eq:u_bound_by_b} shows that
\begin{equation}\label{eq:family_decay}
|u - u_\infty| \leq C(R_0)\|u\|_{L^\infty}\exp\left(-\tfrac{x_3 + \sqrt{x_2^2 + x_3^2}}{2}\right), \:\text{on}\: \left(\mathbb{R}^2\setminus{B_{R_0}}\right)\cap \{x_3\neq -|x|\}.
\end{equation}
Finally, extending over the omitted negative $x_3$-axis, meaning the set $\{x_3=-|x|\}$, by continuity of both sides in the inequality \eqref{eq:family_decay}, yields the desired bound in \eqref{eq:smeared_sup_bound}, so that Lemma \ref{lem:first_decay} is proven.
\end{proof}

\begin{remark}\label{rem:not_barrier}
    In the above proof, we note that the right-hand-side function in \eqref{eq:exp_decay_claim}, which is a bound from the minimization over the $b_{a,R_0}$ family from \eqref{b_familien}, is in fact a supersolution. However, it is not in itself very useful as a nonlinear barrier in the problem, without a separate proof of a limits of the solution at infinity. The above argument establishes this in the process, via the slanted half-space uniform convergence in \eqref{bdry_limit}. Therefore, one could not straightforwardly simplify the above argument by direct use of the right-hand-side function.
\end{remark}

We will now use the supremum bound to obtain the following first improvement on the a priori gradient estimate for exterior solutions to the nonlinear equation.\\

\begin{lem} For $|x|$ sufficiently large, the following a priori gradient estimate holds in \eqref{TSE}:
\begin{equation}\label{eq:first_grad_est}
|\nabla u|\leq C_1(R_0,\|u\|_{L^\infty})(1+|x|^2)^{-1/4}\exp\left(-\tfrac{x_3 + \sqrt{x_2^2 + x_3^2}}{4}\right).
\end{equation}
\end{lem}

Note that this constitutes a first step in the bootstrapping on the gradient decay, and that a stronger estimate, listed in \eqref{eq:G_bound_grad}, will later be proven. 

\begin{proof}
The proof of the claimed estimate \eqref{eq:first_grad_est} will follow from an anisotropic $L^\infty$-interpolation argument, using \cite[Proposition 3.10]{GMM-UniquenessTangentPlanes}, i.e. the non-standard a priori curvature estimate (strengthened Ecker-Huisken estimates) recalled in Proposition \ref{prop:poelsekapsel}, together with the decay obtained above in Lemma \ref{lem:first_decay}.

Seeing however as the decay profiles here are no longer simply inverse powers of the distance function, in order to bootstrap on the decay profiles we will need to verify such non-standard, anisotropic interpolation properties in more detail. Let us for this purpose denote the decay factor in \eqref{eq:smeared_sup_bound} by $\omega$,
\begin{equation}
    \omega(x) = \exp\left(-\tfrac{x_3 + \sqrt{x_2^2 + x_3^2}}{2}\right).
\end{equation}

For interpolation in $L^\infty$, we will be making use of the elementary consequence of Taylor's theorem with explicit remainder term in one real variable known as Landau-Kolmogorov's inequality, which shows that given $f \in C^2([-R, R])$, one has
\begin{equation}\label{Landau-Koolio}
|f'(0)| \le R^{-1}\|f\|_\infty + R \|f''\|_\infty.
\end{equation}
Here, $f$ will be chosen to be $u$ evaluated on rays through a fixed point, such that bounding the corresponding values $f'(0)$ will give the gradient estimate for $u$. To balance the zeroth and second order derivative terms, we will fix $x$ with $|x|\geq \max(4, 2R_0)$ and make the choice of radius:
\[
R:=\sqrt{|x|\omega(x)}\leq\sqrt{|x|}.
\]

Firstly, using that $\mathbb{R}^2\ni x \mapsto |x| + x_3 = -2\log(\omega(x))$ is convex, and then Cauchy-Schwarz, shows that for pairs $x,y\in\mathbb{R}^2$:
\[
(|y| + y_3) - (|x| + x_3) \ge \nabla (|x| + x_3) \cdot (y - x) \ge - |y - x|\left|\tfrac{x}{|x|} + \ee_3\right|,
\]
and we compute that in the final term $\left|\tfrac{x}{|x|} + \ee_3\right|^2 = \frac{2(|x|+x_3)}{|x|}$. Therefore, if $y \in B_R(x)$, that is $|y - x| \le R = \sqrt{|x|\omega(x)}$, we have
\[
(|x| + x_3) - (|y| + y_3) \le \left|\tfrac{x}{|x|} + \ee_3\right| R = \sqrt{2(|x|+x_3) e^{-(|x|+x_3)/2}}.
\]
Using now the elementary bound $ae^{-a/2} \leq 2e^{-1}$ for $a\geq 0$, we see that in fact the decay weight $\omega$ remains comparable in size, explicitly $\omega(y) \le e^{\sqrt{1/e}}\omega(x)$, whenever $y \in B_R(x)$ for $R = \sqrt{|x|\omega(x)}$.
In \eqref{eq:smeared_sup_bound}, labeling the constant by $A=C(R_0)\|u\|_{L^\infty}$, we therefore see that
\begin{equation}
    \sup_{B_R(x)} |u-u_\infty| \leq A e^{\sqrt{1/e}} \omega(x).
\end{equation}

Denoting now in the Hessian estimate \eqref{est:general} the constant by $B=B(R_0,\|u\|_\infty)$, i.e.
\begin{equation}\label{eq:HessBound}
|\Hess u|\leq B/|x|,\quad |x| \geq R_0,
\end{equation}
we find that also
\begin{equation}
    \sup_{B_R(x)} |\Hess u| \leq \frac{2B}{|x|},
\end{equation}
seeing as $R\leq\sqrt{|x|}$, where we recall that we are assuming  $|x|\geq \max(4,R_0)\geq 4$. Finally, inserting these supremum bounds over $B_R(x)$ into the Landau-Kolmogorov inequality \eqref{Landau-Koolio}, for $u$ restricted to all segments of length $2R$ centered at $x$, yields:
\begin{equation}
|\nabla u(x)|\leq \frac{A e^{\sqrt{1/e}} \omega(x)}{\sqrt{\omega(x) |x|}} + \frac{2B \sqrt{\omega(x) |x|}}{|x|} = \left( A e^{\sqrt{1/e}} + 2B \right) \frac{\sqrt{\omega(x)}}{\sqrt{|x|}}.
\end{equation}
This shows the claim in \eqref{eq:first_grad_est}, i.e. that the gradient bound \eqref{eq:first_grad_est} holds with $C_1 = A e^{\sqrt{1/e}} + 2B $, as long as $|x|\geq \max(4, 2R_0)$.
\end{proof}

Therefore, the nonlinear term in \eqref{TSE} also has a stronger bound, now with the very fast decay that follows from combining \eqref{eq:first_grad_est} and \eqref{eq:HessBound}:
\begin{equation}\label{eq:P_term_exp_bound}
\left|P(\nabla u, \operatorname{Hess}u)\right|\leq C_2(R_0,\|u\|_{L^\infty})(1+|x|^2)^{-1}\exp\left(-\tfrac{x_3 + \sqrt{x_2^2 + x_3^2}}{2}\right).
\end{equation}

In fact, the decay obtained thus far is sufficiently strong that we can now reduce the proof of Theorem \ref{thm:optimal_decay} to the following basic lemma for the equation $Lu = f$ on exterior domains:

\begin{lem}\label{lem:Lu_f_decay}
Suppose that $u \in C^2(\mathbb{R}^2\setminus B_{R_0}(0))$ is a solution to
\[
\begin{cases}
Lu = f \quad \text{on} \quad \mathbb{R}^2\setminus B_{R_0}(0),\\
e^{x_3/2}u(x) = o(1),\quad \text{as} \quad |x|\to\infty,     
\end{cases}
\]
where $L = \Delta + \ee_3\cdot\nabla$ and $f \in C^0(\mathbb{R}^2\setminus B_{R_0}(0))$ satisfies
\[
\sup_{x\in\mathbb{R}^2\setminus B_{R_0}(0)}\left|(1+|x|^2)e^{\frac{x_3}{2} + \frac{|x|}{2}}f(x)\right| < \infty.
\]
Then there exists $C=C\left(R_0,\,\|e^{\frac{x_3}{2}}u\|_{\infty},\, \|(1+|x|^2)e^{\frac{x_3}{2} + \frac{|x|}{2}}f\|_\infty\right)$ such that:
\begin{equation}\label{eq:Lu_f_bound}
|u(x)| \le Ce^{-\frac{x_3}{2}}K_0\left(\tfrac{|x|}{2}\right), \quad x\in\mathbb{R}^2\setminus B_{R_0}(0).
\end{equation}
\end{lem}

\begin{proof}
 We may WLOG assume that $R_0\ge 2$. The proof follows by taking a suitable supersolution, e.g. if we denote $M=\|(1+|x|^2)e^{\frac{|x|}{2}+\frac{x_3}{2}}f\|_\infty$,
 and $|x| = |(x_2,x_3)|$
 \[
 w = \alpha\left(\sqrt{\tfrac{R_0}{|x|}} - \tfrac{1}{|x|}\right)e^{- \frac{x_3}{2}-\frac{|x|}{2}},\quad \alpha = \max\left[4M,\, \frac{R_0e^{R_0/2}}{R_0-1}\sup_{\partial B_{R_0}(0)}\left|e^{\frac{x_3}{2}}u\right| \right].
 \]
When $R_0 \geq 2$, this $w = w(x_2,x_3)$ has the properties that $w > 0$, that $w \geq |u|$ on $\partial B_{R_0}(0)$. We calculate the effect of applying $L$ to it:
\[
Lw = -\alpha\left(\tfrac{1}{|x|^3} + \tfrac{1}{2|x|^2} - \tfrac{\sqrt{R_0/|x|}}{4|x|^2}\right)e^{- \frac{x_3}{2}-\frac{|x|}{2}}
\leq -\tfrac{M}{1+|x|^2}e^{- \frac{x_3}{2}-\frac{|x|}{2}},
\]
where we used that $|x| \geq R_0$ also means that $\sqrt{R_0/|x|} \leq 1$.

Now, by $|f(x)| \le M(1+|x|^2)^{-1}e^{-\frac{x_3+|x|}{2}}$, using that $Lu = f$, we thus obtain:
\begin{equation}
L(w \pm u) \leq -M(1+|x|^2)^{-1}e^{- \frac{x_3}{2}-\frac{|x|}{2}} \pm f(x)\leq 0,\quad |x|\geq R_0\geq 2.
\end{equation}
Recalling the assumption $\lim_{|x|\to\infty} e^{\frac{x_3}{2}}u(x) = 0$, so that also $\lim_{|x|\to\infty} e^{\frac{x_3}{2}}(w\pm u)  = 0$, and using the conjugation to the Yukawa operator,
\begin{equation}
\left(\Delta - \tfrac{1}{4}\right)\left(e^{\frac{x_3}{2}}(w\pm u)\right) = e^{\frac{x_3}{2}}L(w\pm u)\leq 0,
\end{equation}
the maximum principle for $\Delta - \tfrac{1}{4}$ applied to the exhaustion of the noncompact domain by annuli now implies that $e^{\frac{x_3}{2}}(w\pm u) \geq 0$ everywhere, or in other words:
\[
|u|\leq w \leq \alpha \sqrt{\tfrac{R_0}{|x|}}e^{-\frac{x_3}{2}-\frac{|x|}{2}}\leq \alpha \sqrt{R_0}C_Ke^{-\frac{x_3}{2}}K_0(|x|/2), \quad\text{on}\quad\mathbb{R}^2\setminus B_{R_0}(0).
\]
Here we took $C_K \geq \left[\inf_{|x|\geq 2} (|x|^{1/2}e^{|x|/2}K_0(|x|/2))\right]^{-1}>0$, i.e. $C_K = 2$ suffices. Thus, \eqref{eq:Lu_f_bound} holds, if we let $$C = 2\sqrt{R_0}\max\left[4\|(1+|x|^2)e^{\frac{x_3}{2}+\frac{|x|}{2}}f\|_\infty,\, \tfrac{R_0e^{R_0/2}}{R_0-1}\|e^{\frac{x_3}{2}}u\|_\infty\right].$$
\end{proof}

Finally, we may now finish the proof of Theorem \ref{thm:optimal_decay}:

\begin{proof}[Proof of Theorem \ref{thm:optimal_decay}]
Letting $f = P(\nabla u, \operatorname{Hess} u)$, and then inserting the estimate
\eqref{eq:P_term_exp_bound} into Lemma \ref{lem:Lu_f_decay}, where we may WLOG assume that $u_\infty = 0$, we see that \eqref{eq:final_G_bound} holds outside a ball of radius $\max(4, 2R_0)$. Possibly increasing the constant, \eqref{eq:final_G_bound} thus holds on the whole exterior domain.

Now by \eqref{eq:final_G_bound}, using \eqref{eq:P_term_exp_bound} and the nonstandard elliptic estimate for $L$ in Lemma \ref{lem:varme_gradient_inhomogent} on capsules of radius 1, with the translator equation \eqref{TSE_split}, the gradient of $u$ also decays at the rate of the Green's function.
\end{proof}

\begin{rem}\label{rem:sharpness}
Regarding sharpness of the gradient bound in \eqref{eq:G_bound_grad}, note that one cannot expect general improvements to it in the form of a multiplier by inverse powers of distance on the right-hand-side $G_L$, seeing as for the Green's function $G_L$ itself, the bound in \eqref{eq:G_bound_grad} is saturated within that class. This is a subtle point, which at first sight might seem to contradict the phenomenon that the global uniform decay $O(|x|^{-1/2})$ of $G_L$ leads to the gradient $\nabla G_L$ decaying as $O(|x|^{-1})$, as a general fact that follows from Theorem \ref{thm:optimal_decay}, Lemma \ref{lem:varme_gradient_inhomogent} and \eqref{eq:P_term_exp_bound}.

That these two perspectives on gain (or lack of gain) of decay when taking derivatives are in fact consistent is a consequence of the anisotropy of the $G_L$ weights. Combining, an anisotropic estimate stronger than both of the above estimates holds:
\begin{equation}
|\nabla u(x)|\leq C\min\left(|x|^{-1},G_L(x)\right), \quad\dist(x,\Omega) > 1.
\end{equation}
\end{rem}

\begin{rem}\label{rem:ray_counterex}
 In \cite{GMM-UniquenessTangentPlanes}, we proved limits along all rays, valid for general collapsed translating solitons, in particular in the exterior problem \cite[Corollary 4.21]{GMM-UniquenessTangentPlanes}. However, note that such properties alone are generally not enough to establish limits at infinity in any strong sense. Namely, for any given $\varepsilon > 0$, the gradient of a smooth function $u_0$ (not required to solve any PDE) may decay as much as $\nabla u_0 = O(|x|^{-1+\varepsilon})$ with zero limits $\lim_{r\to\infty}u_0(r\boldsymbol{\omega})=0$ along all rays $\boldsymbol{\omega}\in\mathbb{S}^1$, and yet $\sup_{|x|=R}|u_0| = 1$ for $R\to\infty$.
 E.g. explicitly take $u_0(x)=(2\,x_2\,|x|\log|x|)/(|x|^{2}+x_2^{2}\log^{2}|x|)$. See also the counterexamples in \cite[Section 5]{GMM-UniquenessTangentPlanes}.
 
\end{rem}

\section{The case of planar wings}

Throughout this section, $\Sigma \subset \mathbb{R}^3$ denotes a complete, connected, embedded, collapsed translator with finite genus, finite entropy, and only
planar wings. We write
$$
\pi:\mathbb{R}^3 \longrightarrow \mathbb{R}^2,\qquad
\pi(x_1,x_2,x_3)=(x_2,x_3),
$$
and, for $R>0$,
$$
\mathcal{C}_R=\{(x_1,x_2,x_3)\in \mathbb{R}^3: x_2^2+x_3^2\leq R^2\},
$$ $$
A_R=\mathbb{R}^2\setminus \overline{B_R(0)}=\{(x_2,x_3)\;:\; |(x_2,x_3)|> R\}.
$$

The first point we make is differential topological. It will be used to say that once the normal can no longer become orthogonal to $\ee_1$, outside a large enough cylinder, properness (which will follow from completeness by the assumption of finite entropy) and embeddedness together rule out non-trivial multi-sheeted coverings over the exterior noncompact annulus:

\begin{prop}\label{prop:graphical_decomposition}
Let $\Sigma\subset\mathbb{R}^3$ be a connected, properly embedded surface without boundary, contained in a slab, i.e. for some $a>0$,
$$
\Sigma\subset\{|x_1|\leq a\}.
$$
Suppose that there exists $R_0>0$ such that
$$
\langle \nu,\ee_1\rangle \neq 0
$$
on $\Sigma\setminus \mathcal{C}_{R_0}$, where $\nu$ is a unit normal to $\Sigma$.
Then $\Sigma\setminus \mathcal{C}_{R_0}$ is a finite union of disjoint graphs
$$
\Gamma_i=\{(u_i(x_2,x_3),x_2,x_3):(x_2,x_3)\in A_{R_0}\},
$$
where each $u_i:A_{R_0}\to\mathbb{R}$ is smooth.
\end{prop}

\begin{proof}
Let $E$ be a connected component of $\Sigma\setminus \mathcal{C}_{R_0}$. We will first show that $\pi|_E$ is a proper local diffeomorphism onto $A_{R_0}$.

At each point $p\in E$, the kernel of $d\pi$ is the line spanned by $\ee_1$.
Therefore $d\pi|_{T_p\Sigma}$ fails to be injective if and only if
$\ee_1\in T_p\Sigma$, equivalently if and only if
$$
\langle \nu(p),\ee_1\rangle=0.
$$
This is excluded by hypothesis. Hence $\pi|_E$ is a local diffeomorphism.

We next prove properness. Let $K\subset A_{R_0}$ be compact. Since $\Sigma$ is
contained in the slab $\{|x_1|\leq a\}$,
$$
(\pi|_E)^{-1}(K)\subset \Sigma\cap([-a,a]\times K).
$$
The set $[-a,a]\times K$ is compact in $\mathbb{R}^3$, and $\Sigma$ is proper.
Thus $\Sigma\cap([-a,a]\times K)$ is compact. Since $E$ is a connected component of $\Sigma\setminus\mathcal{C}_{R_0}$, it is closed in $\Sigma\setminus\mathcal{C}_{R_0}$, which contains $\Sigma\cap([-a,a]\times K)$ because $K\subset A_{R_0}$. Hence $(\pi|_E)^{-1}(K)=E\cap([-a,a]\times K)$ is a closed subset of a compact set, and so is compact.
Hence $\pi|_E$ is proper.

Consequently $\pi|_E:E\to A_{R_0}$ is a proper local diffeomorphism. Its image is
open. It is also closed in $A_{R_0}$, by properness. Since $A_{R_0}$ is connected,
the image is all of $A_{R_0}$. Thus $\pi|_E:E\to A_{R_0}$ is a covering map.

We now prove that this covering has only one sheet. Since $\pi|_E$ is proper, each
fiber is compact, and since $\pi|_E$ is a local diffeomorphism, each fiber is discrete.
Thus every fiber is finite, and the number of points in the fiber is constant on
$A_{R_0}$. We call this number $n$.

For $q=(x_2,x_3)\in A_{R_0}$, the points of $\pi^{-1}(q)\cap E$ lie on the line
parallel to $\ee_1$ through $q$. We may therefore order them by their $x_1$-coordinate:
$$
\pi^{-1}(q)\cap E
=
\{(u_1(q),x_2,x_3),\ldots,(u_n(q),x_2,x_3)\},
$$
with
$$
u_1(q)<u_2(q)<\cdots<u_n(q).
$$
Because $\pi|_E$ is a covering map, the sheets are locally smooth graphs. Since
$\Sigma$ is embedded, two such local sheets cannot cross; hence their order in the
$x_1$-direction cannot change. Therefore $u_1,\ldots,u_n$ are globally defined
smooth functions on $A_{R_0}$.

It follows that
$$
E=\Gamma_1\cup\cdots\cup\Gamma_n,
$$
where
$$
\Gamma_j=\{(u_j(x_2,x_3),x_2,x_3):(x_2,x_3)\in A_{R_0}\}.
$$
The sets $\Gamma_j$ are pairwise disjoint, nonempty, and both open and closed in
$E$. Since $E$ is connected, we must have $n=1$. Hence each connected component
$E$ is a single graph over $A_{R_0}$.

Finally, there are only finitely many such components. Fix $q_0\in A_{R_0}$. Each
component of $\Sigma\setminus \mathcal{C}_{R_0}$ meets the line $\pi^{-1}(q_0)$ in
exactly one point. Thus the number of components equals the cardinality of
$$
\Sigma\cap\pi^{-1}(q_0).
$$
This set is compact, by properness and the slab bound, and discrete, by the local
diffeomorphism property. Hence it is finite.
\end{proof}

We now apply this proposition to the present translator. By the asymptotic
description of collapsed translators with finite genus and finite entropy \cite[Sections 6 and 7]{GMM-UniquenessTangentPlanes}, and using
that all wings are planar, there exists $R_0>0$ such that
$$
\langle \nu,\ee_1\rangle\neq 0
$$
on $\Sigma\setminus\mathcal{C}_{R_0}$. Hence, Proposition
\ref{prop:graphical_decomposition} gives a finite decomposition
$$
\Sigma\setminus\mathcal{C}_{R_0}
=
E_1\cup\cdots\cup E_N,
$$
where each end is a smooth horizontal graph
$$
E_j=\{(u_j(x_2,x_3),x_2,x_3):(x_2,x_3)\in A_{R_0}\}.
$$
Since $\Sigma$ is contained in a slab, each $u_j$ is globally bounded. Moreover, each $u_j$
satisfies the horizontal graphical translator equation \eqref{TSE}.

Thus, Theorem \ref{thm:optimal_decay} and the subsequent gradient estimate
\eqref{eq:G_bound_grad} apply to every graphical end $E_j$.

We shall need the following consequence.

\begin{lem}\label{lem:graph_flux_decay}
Let $u:A_{R_0}\to\mathbb{R}$ be a bounded solution of the horizontal graphical 
translator equation \eqref{TSE}, and let
$$
X(x_2,x_3)=(u(x_2,x_3),x_2,x_3)
$$
be its graph parametrization. For $R>R_0$, set
$$
S_R=\{(x_2,x_3):x_2^2+x_3^2=R^2\}.
$$
Let $\eta$ be the outward unit conormal to $X(S_R)$ as a curve in the graph. Then
$$
\left|\int_{X(S_R)} \langle \eta,\ee_3\rangle\,ds_\Sigma\right|
\leq C R^{-1/2},\quad\mathrm{as}\quad R\to\infty.
$$
\end{lem}

\begin{proof}
Write
$$
W=\sqrt{1+|\nabla u|^2},
\qquad
(\mu_2,\mu_3)=\mu=\frac{(x_2,x_3)}{R}
$$
for the Euclidean outward unit normal to $S_R$ in the parameter domain, and denote
$$
u_\mu=\langle \nabla u,\mu\rangle,\quad u_3 = \langle \nabla u,\ee_3\rangle.
$$
A direct computation for horizontal graphs gives
$$
\int_{X(S_R)} \langle \eta,\ee_3\rangle\,ds_\Sigma
=
\int_{S_R}
\left(
W\mu_3-\frac{u_\mu u_3}{W}
\right)\,ds_0,
$$
where $ds_0$ denotes Euclidean arc length in the $(x_2,x_3)$-plane.

The next couple of steps are similar to the argument in \cite{Lynch-Tinaglia,Neves-Tian}. Since
$$
\int_{S_R}\mu_3\,ds_0=0,
$$
we may subtract the leading term and obtain
$$
\int_{X(S_R)} \langle \eta,\ee_3\rangle\,ds_\Sigma
=
\int_{S_R}
\left(
(W-1)\mu_3-\frac{u_\mu u_3}{W}
\right)\,ds_0.
$$
Using that on $S_R$,
\begin{equation}\label{eq:W_diff}
0\leq W-1 \leq \tfrac{1}{2}|\nabla u|^2,
\qquad
\left|\frac{u_\mu u_3}{W}\right|\leq |\nabla u|^2,
\end{equation}
we get
$$
\left|\int_{X(S_R)} \langle \eta,\ee_3\rangle\,ds_\Sigma\right|
\leq
\frac{3}{2}\int_{S_R}|\nabla u|^2\,ds_0.
$$

By the a priori gradient estimate \eqref{eq:G_bound_grad},
$$
|\nabla u(x)|\leq C G_L(x),
\qquad
G_L(x)=\frac{1}{2\pi}e^{-x_3/2}K_0(|x|/2).
$$
On $S_R$, write $x_3=R\sin\theta$. Then
$$
|\nabla u|^2
\leq
C e^{-R\sin\theta}K_0(R/2)^2.
$$
Therefore
$$
\int_{S_R}|\nabla u|^2\,ds_0
\leq
C R K_0(R/2)^2\int_0^{2\pi}e^{-R\sin\theta}\,d\theta.
$$
The last integral is equal to $2\pi I_0(R)$, where $I_0$ is
the modified Bessel function of the first kind. Since
$$
K_0(R/2) = O\left(\sqrt{\frac{\pi}{R}}e^{-R/2}\right),
\qquad
I_0(R)=O\left(\frac{e^R}{\sqrt{2\pi R}}\right),
$$
we conclude that
$$
\int_{S_R}|\nabla u|^2\,ds_0
\leq
C R K_0(R/2)^2 I_0(R)=O(R^{-1/2}),
$$
which proves the desired estimate.
\end{proof}

We finally recall the flux identity of Neves--Tian.

\begin{lem}[\cite{Neves-Tian}]\label{lem:neves_tian}
Let $\Sigma$ be a compact translator with smooth boundary. Then
$$
\int_\Sigma |\vec H|^2\,dA
=
\int_{\partial \Sigma}\langle \eta,\ee_3\rangle\,d\sigma,
$$
where $\eta$ denotes the outward unit conormal to $\partial \Sigma$ on $\Sigma$.
\end{lem}

We can now restate and prove Theorem (I) from the introduction:

\begin{thm}\label{thm:planar_wings}
Let $\Sigma$ be a complete, connected, embedded, collapsed translator with finite genus, finite entropy, and only planar wings. Then $\Sigma$ is a vertical plane.
\end{thm}

\begin{proof}
Choose $R_0$ as above, so that
$$
\Sigma\setminus\mathcal{C}_{R_0}
=
E_1\cup\cdots\cup E_N
$$
is a finite union of horizontal graphical ends. For $R>R_0$, define the compact
region
$$
M_R=\Sigma\cap\mathcal{C}_R.
$$
Its boundary is the disjoint union
$$
\partial M_R=\Gamma_1(R)\cup\cdots\cup\Gamma_N(R),
$$
where
$$
\Gamma_j(R)=E_j\cap \partial\mathcal{C}_R.
$$

Applying Lemma \ref{lem:neves_tian} to $M_R$, we obtain
$$
\int_{M_R}|\vec H|^2\,dA
=
\sum_{j=1}^N
\int_{\Gamma_j(R)}\langle \eta,\ee_3\rangle\,d\sigma.
$$
Each $\Gamma_j(R)$ is the lift of the circle $S_R$ by the graph $u_j$. Hence Lemma
\ref{lem:graph_flux_decay} gives
$$
\left|
\int_{\Gamma_j(R)}\langle \eta,\ee_3\rangle\,d\sigma
\right|
\leq C_jR^{-1/2}.
$$
Therefore
$$
\int_{M_R}|\vec H|^2\,dA
\leq
CR^{-1/2}.
$$
Letting $R\to\infty$, we obtain
$$
\int_\Sigma |\vec H|^2\,dA=0.
$$
Thus $\vec H\equiv 0$ on $\Sigma$.

Since $\Sigma$ is a translator, $\vec H=\ee_3^\perp$. Hence $\ee_3^\perp\equiv 0$,
so $\ee_3$ is tangent to $\Sigma$ everywhere. Consequently, it is also minimal, and thus a vertical plane.
\end{proof}

\begin{rem}\label{rem:Neves-Tian}
We note that the previous proof carries over to a broader class of immersed translators of finite entropy and finite genus. Indeed, the only additional requirement is that the ends of the translator are collapsed graphs, in the sense of the definition below. In particular, this implies that one cannot construct Jorge–Meeks surfaces in the translating setting whose ends are graphically collapsed. 

Notice that if $\Sigma$ is an immersed surface with finite entropy, finite genus and graphically collapsed ends, then it is not necessarily collapsed, in the sense that $\Sigma$ is not necessarily contained in a vertical slab of $\R^3$.
However, the hypothesis about finite entropy is crucial. Namely, Hoffman-Mart\'{\i}n-White \cite{HMW-1} constructed a family of periodic, simply connected translators, called translating helicoids (see Figure \ref{fig:helicoids}), which have infinitely many planar wings, so that they have infinite entropy.
\end{rem}

\begin{figure}
    \centering
    \includegraphics[width=0.65\linewidth]{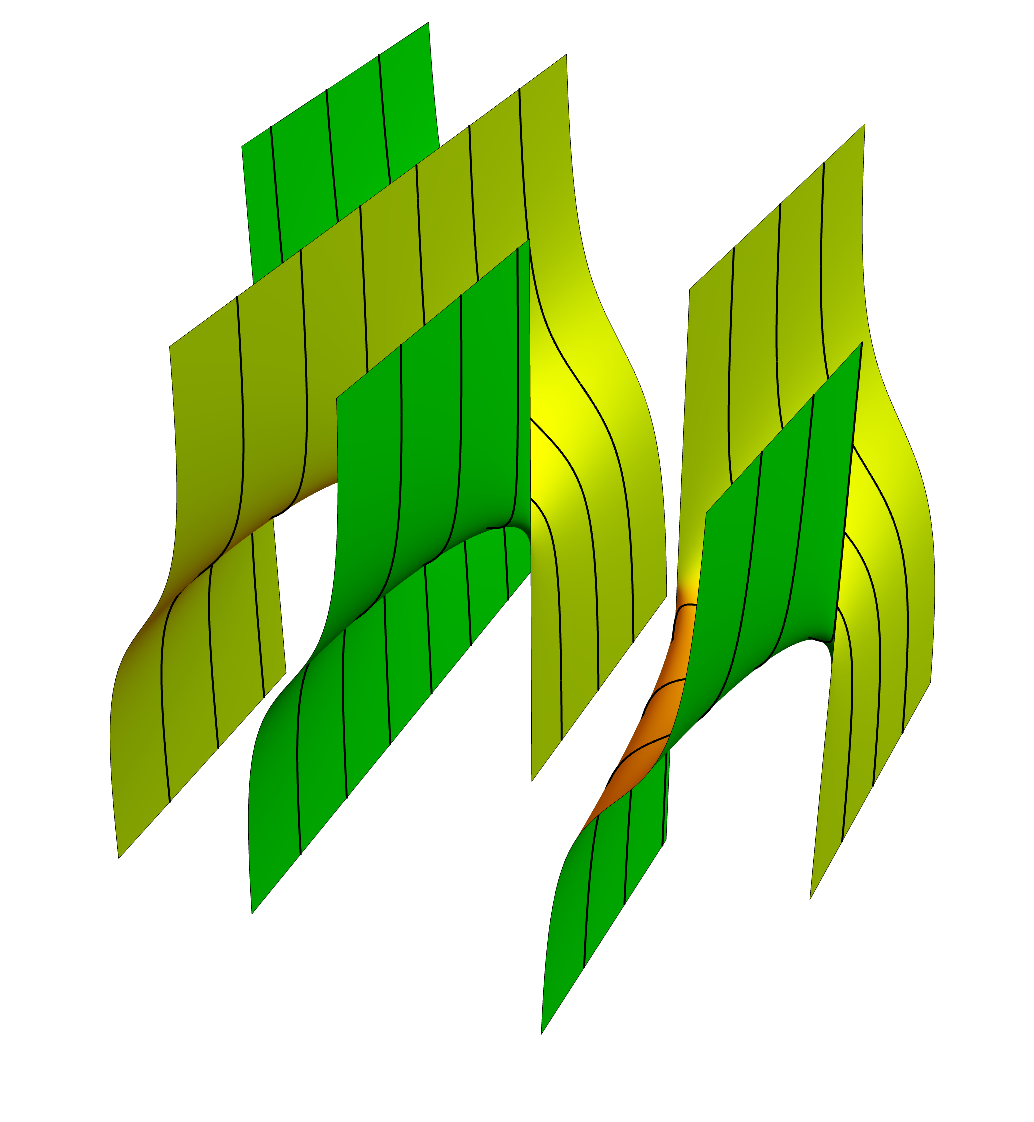}
    \caption{Translating helicoid of Hoffman-Mart\'{\i}n-White. See Remark \ref{rem:Neves-Tian}. The picture on the left is a fundamental piece, bounded by 2 vertical lines. The complete surface is produced by successive Schwarz reflections about these vertical lines.}
    \label{fig:helicoids}
\end{figure}

\begin{dfn}[Collapsed graphical ends]\label{dfn:collapsed_graphical_end}
    We say that an end $E$ of $\Sigma^2\subset\mathbb{R}^3$ is a horizontal collapsed graphical end, if there exists a horizontal unit vector $v\perp \ee_3$ and a smooth bounded function $\varphi:[v]^\perp\setminus\Omega\to\R,$ where $\Omega\subseteq[v]^\perp$ is a smooth bounded planar domain, such that
    \[
    E=\operatorname{Graph}_{v,\Omega}(\varphi):= \{\,p+\varphi(p) v \; : \; p\in[v]^\perp\setminus\Omega\,\}.
    \]
    In the rest of the paper we generally arrange by invariance of \eqref{TSE_geom} that $0\in\Omega$ and $v=\ee_1$, as in \eqref{eq:horiz_graph}.
\end{dfn}

\section{Translators of entropy two}\label{sec:entropy-two}

Let $\Sigma\subset\mathbb{R}^3$ be a complete, connected, collapsed, embedded translator, finite genus, and finite entropy. 

Collapsedness, also known as finite width, means that $$ \Sigma \subset (-w/2,w/2) \times \R^2, \quad \mbox{for $w=\width(\Sigma) > 0$ (infimum of such $w$)}$$
For $t>0$ sufficiently large,
denote by
$$
\omega_P^+(\Sigma),\qquad \omega_G^+(\Sigma)
$$
the numbers of planar and grim-reaper-type wings, respectively, in
$\{x_2>t\}$, and similarly by
$$
\omega_P^-(\Sigma),\qquad \omega_G^-(\Sigma)
$$
the corresponding numbers in $\{x_2<-t\}$.

Recall from \cite[Theorem 8.2]{GMM-AJM} that
\begin{equation}
\lambda(\Sigma)=
\omega_P^-(\Sigma)+2\omega_G^-(\Sigma) =
\omega_P^+(\Sigma)+2\omega_G^+(\Sigma).
\label{eq:entropy-left-right}
\end{equation}
Moreover, planar wings occur naturally in left--right pairs. More precisely, the
decomposition into downward graphical regions described in
\cite[Proposition 2.6]{GMM-UniquenessTangentPlanes} shows that every such region
contains precisely one left planar wing and one right planar wing. Consequently,
\begin{equation}
\omega_P^-(\Sigma)
= \omega_P^+(\Sigma)=
\frac{\omega_P(\Sigma)}{2}.
\label{eq:planar-pairing}
\end{equation}

In \cite{GMM-UniquenessTangentPlanes}, the authors also proved the following:
\begin{thm}[Theorem F, \cite{GMM-UniquenessTangentPlanes}]\label{thmCC}
    Let $\Sigma$ be  a complete, embedded, collapsed translator with finite genus and $\lambda(\Sigma)=2$. Assume that  the limit, as $t \to +\infty$, of $\Sigma+t \ee_3$ is the empty set. Then $\Sigma$ is either a (possibly tilted) grim reaper or a $\Delta$-wing. In particular, if the width of the slab containing $\Sigma$ is $\pi$, then $\Sigma$ is the standard grim reaper cylinder.
\end{thm}

Now, using Theorem \ref{thm:planar_wings}, we can complete the classification of collapsed translators with finite genus and entropy two. The point is that the wing-counting formula \eqref{eq:entropy-left-right} leaves only two possibilities: either all the wings are planar, or all of them are of grim reaper
type. The first possibility is ruled out by Theorem~\ref{thm:planar_wings}, while the second one is resolved by Theorem \ref{thmCC}.

\begin{thm}
\label{th:entropy-two-classification}
Let $\Sigma\subset\mathbb{R}^3$ be a complete, connected, collapsed, embedded translating
soliton with finite genus. If the entropy satisfies
$$
\lambda(\Sigma)=2,
$$
then $\Sigma$ is either a possibly tilted grim reaper cylinder or a
$\Delta$-wing.

In particular, if
$$
\width(\Sigma)=\pi,
$$
then $\Sigma$ is the standard grim reaper cylinder.
\end{thm}

\begin{proof}
Set
$$
p:=\omega_P^-(\Sigma)=\omega_P^+(\Sigma).
$$
Since $\lambda(\Sigma)=2$, equation \eqref{eq:entropy-left-right} gives
$$
2=p+2\omega_G^-(\Sigma)
  =p+2\omega_G^+(\Sigma).
$$
Thus $p$ is even and
$$
p\in\{0,2\}.
$$
Suppose first that $p=2$. Then
$$
\omega_G^-(\Sigma)=\omega_G^+(\Sigma)=0.
$$
By \eqref{eq:planar-pairing}, $\Sigma$ has four planar wings and no wings of
grim reaper type. As all the wings of $\Sigma$ are planar, Theorem~\ref{thm:planar_wings} then implies that $\Sigma$ is a vertical plane.
But a vertical plane has entropy one, contradicting the hypothesis
$\lambda(\Sigma)=2$.

Hence $p=0$. It follows that
$$
\omega_P(\Sigma)=0
$$
and
$$
\omega_G^-(\Sigma)=\omega_G^+(\Sigma)=1.
$$

Thus $\Sigma$ has exactly two wings, both of grim reaper type. Equivalently, by 
\cite[Theorem A]{GMM-UniquenessTangentPlanes},
$$
\Sigma+t\ee_3\longrightarrow\varnothing
\qquad\text{as }t\to+\infty.
$$

We may therefore apply Theorem \ref{thmCC} to conclude that $\Sigma$ is either a possibly tilted grim reaper cylinder or a $\Delta$-wing.

Finally, when $\width(\Sigma)=\pi$, the same theorem gives that $\Sigma$ is the standard grim reaper cylinder.
\end{proof}

\begin{rem}
Theorem~\ref{th:entropy-two-classification} removes the additional assumption
$$
\Sigma+t\ee_3\longrightarrow\varnothing
\qquad\text{as }t\to+\infty
$$
from Theorem \ref{thmCC} above. Indeed, that assumption was precisely what
excluded planar wings. Theorem~\ref{thm:planar_wings} shows that the only
finite-genus, finite-entropy collapsed translator whose wings are all of planar type as in Remark \ref{rem:planar_wings} is a vertical plane, and the entropy $\lambda = 2$ condition rules this out.

Thus there are no finite-genus ``higher-topology'' versions of the
grim reaper cylinder or of the $\Delta$-wing at entropy two.
\end{rem}

Combining this with the entropy quantization theorem gives the following convenient
low-entropy formulation, which disposes of the extra assumption of simple connectedness present in the previously best classification results for this entropy range in \cite{ChiniThesis}\cite[Corollary 2]{Chini}\cite[Theorem 10.1]{GMM-AJM}.

\begin{cor}
\label{cor:entropy-at-most-two}
Let $\Sigma^2\subset\mathbb{R}^3$ be a collapsed, complete, connected, embedded translating soliton with finite genus, and
$$
\lambda(\Sigma)<3.
$$
Then precisely one of the following holds:
\begin{enumerate}
\item $\lambda(\Sigma)=1$ and $\Sigma$ is a vertical plane;
\item $\lambda(\Sigma)=2$ and $\Sigma$ is a possibly tilted grim reaper
cylinder or a $\Delta$-wing.
\end{enumerate}
\end{cor}

\begin{proof}
For collapsed translators of finite genus and finite entropy, the entropy is a
positive integer $\lambda(\Sigma)\in\mathbb{N}$ by \cite{GMM-AJM}. If $\lambda(\Sigma)=1$, the wing
formula \eqref{eq:entropy-left-right} with Theorem \ref{thm:planar_wings}, or alternatively the standard entropy-one rigidity, gives a vertical
plane. If $\lambda(\Sigma)=2$, the assertion follows from
Theorem~\ref{th:entropy-two-classification}.
\end{proof}

\section{Horizontal fluxes of planar ends}
\label{sec:horizontal-flux}

In this section we record some consequences of the sharp decay estimates
for bounded exterior solutions obtained in Section~2. A planar graphical
end carries two natural weighted horizontal fluxes, i.e. fluxes orthogonal to $\ee_3$. The flux component normal
to the asymptotic vertical plane is conserved and may be nonzero, whereas
the component tangent to that plane and orthogonal to $\ee_3$ necessarily vanishes.

We also explain how a nonzero normal flux provides an obstruction to improving the decay estimate of Theorem~\ref{thm:optimal_decay} from $O(G_L)$ to
$o(G_L)$, and give a uniqueness result (flatness) for faster-than-Green's decay. Finally, we compute the corresponding linear fluxes explicitly for
the separated solutions of the linearized equation.

Let
$$
X \colon A_{R_0}\longrightarrow \mathbb{R}^3,
\qquad
X(x_2,x_3)=\bigl(u(x_2,x_3),x_2,x_3\bigr),
$$
be a bounded horizontal graphical translator, where
\[
A_{R_0}=\mathbb{R}^2\setminus B_{R_0}(0).
\]
Recall that, using the notation 
$
W=\sqrt{1+|\nabla u|^2}
$, then the horizontal graphical translator equation is
\[
\operatorname{div}\left(\frac{\nabla u}{W}\right)
=
-\frac{\partial_3u}{W}
\quad \Longleftrightarrow \quad
\operatorname{div}\left(
e^{x_3}\frac{\nabla u}{W}
\right)=0.
\]

For $R>R_0$, let
\[
S_R=\partial B_R(0)\subset\mathbb{R}^2,
\]
let $\mu=(\mu_2,\mu_3)$ be its outward Euclidean unit normal, and let
$\eta$ denote the outward unit conormal to $X(S_R)$ in the graph.
We write
\[
u_\mu=\langle\nabla u,\mu\rangle.
\]

\subsection{The normal horizontal flux}

A direct computation gives
\begin{equation}
\langle\eta,\ee_1\rangle\,ds_\Sigma
=
\frac{u_\mu}{W}\,ds_0,
\label{eq:normal-horizontal-conormal}
\end{equation}
where $ds_0$ denotes Euclidean arclength on $S_R$. We therefore
define
\begin{equation}
\mathcal{F}_1(R)
:=
\int_{X(S_R)}
e^{x_3}\langle\eta,\ee_1\rangle\,ds_\Sigma
=
\int_{S_R}
e^{x_3}\frac{u_\mu}{W}\,ds_0,
\label{eq:normal-horizontal-flux}
\end{equation}
which is finite owing to Theorem \ref{thm:optimal_decay}.

Equation~\eqref{eq:weighted-translator-intro} immediately implies that
$\mathcal{F}_1(R)$ is independent of $R$. Indeed, if
$R_0<R_1<R_2$, the divergence theorem applied to
$B_{R_2}(0)\setminus B_{R_1}(0)$ gives
\[
\mathcal{F}_1(R_2)=\mathcal{F}_1(R_1).
\]
Thus each planar graphical end $E$ carries a well-defined scalar
invariant
\begin{equation}
q(E):=\mathcal{F}_1(R),
\label{eq:horizontal-charge}
\end{equation}
which we call its normal horizontal flux.

More generally, suppose that $u$ is a bounded solution of the
horizontal graphical translator equation on
\[
\mathbb{R}^2\setminus\Omega,
\]
where $\Omega\subset\mathbb{R}^2$ is a smooth bounded domain, and let
$\nu_\Omega$ denote the outward unit normal to $\partial\Omega$.
Applying the divergence theorem to $B_R(0)\setminus\Omega$ yields the identity
\begin{equation}
q(E)
=
\int_{\partial\Omega}
e^{x_3}
\frac{u_{\nu_\Omega}}{W}\,ds_0.
\label{eq:inner-boundary-horizontal-flux}
\end{equation}
Thus the horizontal flux measured at infinity is completely determined
by the Cauchy data on any compact cross-section of the end.

\subsection{Vanishing of the transverse horizontal flux}

There is a second conserved flux, corresponding to translations in the
$\ee_2$-direction. A direct computation gives
\begin{equation}
\langle\eta,\ee_2\rangle\,ds_\Sigma
=
\left(
W\mu_2-\frac{u_\mu\,\partial_2u}{W}
\right)ds_0.
\label{eq:transverse-horizontal-conormal}
\end{equation}
We define
\begin{equation}
\mathcal{F}_2(R)
:=
\int_{X(S_R)}
e^{x_3}\langle\eta,\ee_2\rangle\,ds_\Sigma.
\label{eq:transverse-horizontal-flux}
\end{equation}

This quantity is also independent of $R$. One way to see this is to recall that the graphical translator equation is the Euler--Lagrange
equation for the weighted area functional
\[
\int e^{x_3}W\,dx_2\,dx_3,
\]
which is invariant under translations in the $x_2$-direction.
Equivalently, the vector field
\[
e^{x_3}
\left(
W \ee_2-\frac{\partial_2u}{W}\nabla u
\right)
\]
is divergence-free.

We now show that this second flux must vanish. From
\eqref{eq:transverse-horizontal-conormal},
\[
\mathcal{F}_2(R)
=
\int_{S_R}
e^{x_3}
\left(
W\mu_2-\frac{u_\mu\,\partial_2u}{W}
\right)ds_0.
\]
Since
\begin{equation}
\int_{S_R}e^{x_3}\mu_2\,ds_0=0,
\label{eq:mu2-zero}
\end{equation}
we can subtract the leading term and obtain
\[
\mathcal{F}_2(R)
=
\int_{S_R}
e^{x_3}
\left(
(W-1)\mu_2-\frac{u_\mu\,\partial_2u}{W}
\right)ds_0.
\]
Using (as in \eqref{eq:W_diff})
\[
\left|
\frac{u_\mu\,\partial_2u}{W}
\right|
\leq|\nabla u|^2,\quad |W-1|\leq C|\nabla u|^2,
\]
we find
\begin{equation}
|\mathcal{F}_2(R)|
\leq
C\int_{S_R}e^{x_3}|\nabla u|^2\,ds_0.
\label{eq:transverse-flux-estimate}
\end{equation}

By the gradient estimate following Theorem~\ref{thm:optimal_decay},
\[
|\nabla u(x)|
\leq C G_L(x),
\qquad
G_L(x)
=
\frac{1}{2\pi}
e^{-x_3/2}
K_0\left(\frac{|x|}{2}\right).
\]
Hence, on $S_R$,
\[
e^{x_3}|\nabla u|^2
\leq
C K_0\left(\frac{R}{2}\right)^2,
\]
and therefore
\[
|\mathcal{F}_2(R)|
\leq
CRK_0\left(\frac{R}{2}\right)^2.
\]
Since
\[
K_0\left(\frac{R}{2}\right)
=O\left(
\sqrt{\frac{\pi}{R}}\,e^{-R/2}\right),
\]
we obtain
\[
|\mathcal{F}_2(R)|=O(e^{-R}).
\]
As $\mathcal{F}_2(R)$ is independent of $R$, it follows that
\[
\mathcal{F}_2(R)=0
\]
for every $R>R_0$.

We have proved the following.

\begin{prop}
\label{prop:horizontal-flux}
Let $E$ be a horizontal collapsed graphical end of a translator, as in \ref{dfn:collapsed_graphical_end}. Then
its weighted horizontal flux is parallel to the direction normal to its
asymptotic vertical plane. More precisely,
\[
\int_{X(S_R)}
e^{x_3}\langle\eta,\ee_1\rangle\,ds_\Sigma
=
q(E),
\]
where $q(E)$ is independent of $R$, whereas
\[
\int_{X(S_R)}
e^{x_3}\langle\eta,\ee_2\rangle\,ds_\Sigma
=
0.
\]
\end{prop}

\subsection{An obstruction to improved decay}

The normal horizontal flux gives a simple obstruction to improving the
sharp decay estimate of Theorem~\ref{thm:optimal_decay}.

First observe that
\begin{equation}
\int_{S_R}e^{x_3}G_L\,ds_0
=
R K_0\left(\frac{R}{2}\right)
I_0\left(\frac{R}{2}\right),
\label{eq:weighted-green-integral}
\end{equation}
where $I_0$ is the modified Bessel function of the first kind. Indeed,
writing
\[
x_2=R\cos\theta,
\qquad
x_3=R\sin\theta,
\]
we have
\[
\begin{aligned}
\int_{S_R}e^{x_3}G_L\,ds_0
&=
\frac{R}{2\pi}
K_0\left(\frac{R}{2}\right)
\int_0^{2\pi}e^{(R/2)\sin\theta}\,d\theta
\\
&=
R K_0\left(\frac{R}{2}\right)
I_0\left(\frac{R}{2}\right).
\end{aligned}
\]
Using
\[
I_0(s)=\frac{e^s}{\sqrt{2\pi s}}+o\left(\frac{e^s}{\sqrt{2\pi s}}\right),
\qquad
K_0(s)=\sqrt{\frac{\pi}{2s}}\,e^{-s}+o\left(\sqrt{\frac{\pi}{2s}}\,e^{-s}\right),\quad \mathrm{as}\quad s\to\infty,
\]
we obtain
\begin{equation}
\lim_{R\to\infty}
\int_{S_R}e^{x_3}G_L\,ds_0
=1.
\label{eq:weighted-green-limit}
\end{equation}
Consequently, if
\begin{equation}
|\nabla u(x)|=o(G_L(x))
\qquad\text{as }|x|\to\infty
\label{eq:little-o-gradient}
\end{equation}
uniformly at infinity, then
\[
\begin{aligned}
|q(E)|
&\leq
\int_{S_R}
e^{x_3}\frac{|\nabla u|}{W}\,ds_0
\\
&\leq
\int_{S_R}e^{x_3}|\nabla u|\,ds_0
=o(1)
\end{aligned}
\]
as $R\to\infty$. Hence $q(E)=0$.

\begin{cor}
\label{cor:flux-obstruction}
Let $E$ be a bounded horizontal graphical end. If $q(E)\neq0$, then
the gradient estimate
\[
|\nabla u|=O(G_L)
\]
cannot be improved to
\[
|\nabla u|=o(G_L)
\]
uniformly at infinity.
\end{cor}
Thus a nonzero horizontal flux detects planar ends for which the
Green's function scale occurring in Theorem~\ref{thm:optimal_decay}
cannot be replaced by a strictly smaller uniform order of decay.

An interesting consequence of this fact is the following, where we note that the case where the limit of $u$ at infinity $u_\infty\neq 0$ differs from the boundary value is the interesting case:
\begin{prop}
\label{prop:flux-obstruction-rewritten}
Let $E$ be a horizontal collapsed graphical end as in Definition \ref{dfn:collapsed_graphical_end}, such that $\partial E\subseteq\{x_1=0\}$ is compact. Suppose that the bounded function $u$ representing $E$ satisfies
\begin{equation}\label{extra_decay}
|\nabla u|=o(G_L).
\end{equation}
Then $u = 0$, so that $E$ is a flat plane.
\end{prop}
\begin{proof}

Firstly, existence of the limit follows from Theorem \ref{thm:optimal_decay}:
\begin{equation}\label{eq:u_inf}
    \lim_{|x|\to\infty}u(x)=u_\infty.
\end{equation}
 If $u_\infty = 0$, then by the maximum principle, $u=0$, even without using the extra decay assumption \eqref{extra_decay}. If not, we may assume $u_\infty > 0$. By the maximum principle,
\begin{equation}\label{eq:u_between}
\inf_{\R^2\setminus \Omega}u=0<u(p)<\sup_{\R^2\setminus \Omega} u=u_\infty,\quad\mathrm{for}\;\mathrm{all}\quad p\in\R^2\setminus \overline{\Omega}.
\end{equation}

On $\partial\Omega$,  $|\nabla u| > 0$ follows from $u > 0$ with $u_{|\partial \Omega} = 0$ by Hopf's lemma \cite{Serrin70} for \eqref{TSE}. The unit normal to $\partial \Omega$ that points outwards (towards $\mathbb{R}^2\setminus \Omega$) is given by $\nu=\frac{\nabla u}{|\nabla u|}$, so that $u_\nu = |\nabla u|$. Therefore, the assumption that $|\nabla u|=o(G_L)$ gives $q(E)=0$ in \eqref{eq:inner-boundary-horizontal-flux}, so as shown above,
\begin{equation}\label{zero_flux} \int_{\partial \Omega}e^{x_3}\frac{|\nabla u|}{W}ds_0=0,
\end{equation}
a contradiction. Therefore, $u$ must vanish identically.
\end{proof}

\subsection{Linearized mode analysis}

The preceding observation is naturally reflected in the linearized
equation. Recall that
\[
L=\Delta+\partial_{x_3}.
\]
If
\[
u=e^{-x_3/2}v,
\]
then
\[
Lu
=
e^{-x_3/2}
\left(\Delta v-\frac14v\right).
\]
Hence, as already used in \cite{GMM-UniquenessTangentPlanes} and in the proof of Lemma \ref{lem:Lu_f_decay}, the conjugation gives that $L$-harmonicity of $u$ is equivalent to Yukawa harmonicity of $v$:
\[
Lu=0
\qquad\Longleftrightarrow\qquad
\left(\Delta-\tfrac14\right)v=0.
\]

Writing
\[
x_2=r\cos\theta,
\qquad
x_3=r\sin\theta,
\]
separation of variables gives the exterior solutions
\[
u_n^c(r,\theta)
=
e^{-x_3/2}
K_n\left(\frac r2\right)\cos(n\theta),
\qquad n\geq0,
\]
and
\[
u_n^s(r,\theta)
=
e^{-x_3/2}
K_n\left(\frac r2\right)\sin(n\theta),
\qquad n\geq1.
\]
For each fixed $n$,
\[
K_n(s)
=
\sqrt{\frac{\pi}{2s}}\,e^{-s}
\left(1+O(s^{-1})\right),
\]
and therefore
\[
\frac{K_n(r/2)}{K_0(r/2)}
\longrightarrow1,\quad\mathrm{as}\quad r\to\infty.
\]
Thus all these modes have the same order of decay as $G_L$. In
particular, the bound
\[
|u|\leq C G_L
\]
does not imply that $u/G_L$ converges to a constant.

It is instructive to compute the horizontal flux of these modes explicitly. For $u$ an $L$-harmonic function, i.e. solving $Lu=0$, define the linearized normal flux by
\begin{equation}
q_{\mathrm{lin}}(u)
:=
\int_{S_R}e^{x_3}\partial_\mu u\,ds_0,
\label{eq:linear-flux}
\end{equation}
where $\mu$ denotes the outward unit normal. Since
\[
\operatorname{div}(e^{x_3}\nabla u)=e^{x_3}Lu=0,
\]
the quantity $q_{\mathrm{lin}}(u)$ is then independent of $R$.

Set $s=R/2$. For the cosine modes,
\[
\begin{aligned}
q_{\mathrm{lin}}(u_n^c)
&=
s K_n'(s)
\int_0^{2\pi}
e^{s\sin\theta}\cos(n\theta)\,d\theta
\\
&\quad
-
s K_n(s)
\int_0^{2\pi}
e^{s\sin\theta}
\sin\theta\cos(n\theta)\,d\theta.
\end{aligned}
\]
Using
\[
\int_0^{2\pi}
e^{s\sin\theta}\cos(n\theta)\,d\theta
=
2\pi I_n(s)\cos\left(\frac{n\pi}{2}\right)
\]
and differentiating this identity with respect to $s$, we obtain
\[
q_{\mathrm{lin}}(u_n^c)
=
2\pi s
\left(
I_n(s)K_n'(s)-I_n'(s)K_n(s)
\right)
\cos\left(\frac{n\pi}{2}\right).
\]
The Wronskian identity
\[
I_n(s)K_n'(s)-I_n'(s)K_n(s)
=
-\frac1s
\]
therefore gives
\begin{equation}
q_{\mathrm{lin}}(u_n^c)
=
-2\pi\cos\left(\frac{n\pi}{2}\right).
\label{eq:cos-mode-flux}
\end{equation}

Similarly,
\begin{equation}
q_{\mathrm{lin}}(u_n^s)
=
-2\pi\sin\left(\frac{n\pi}{2}\right).
\label{eq:sin-mode-flux}
\end{equation}

These formulas illustrate both the usefulness and the limitations of
the horizontal flux. Different angular modes generally have the same
order as the Green's function at infinity, so the flux cannot determine the full leading asymptotic behavior. Nevertheless, it extracts a nontrivial scalar quantity from that behavior.

Namely, let a sum of such modes be given,
\[
u
=
\sum_{n=0}^{N}
\left(
a_nu_n^c+b_nu_n^s
\right),
\]
where $b_0=0$, and set
\[
A(\theta)
=
\sum_{n=0}^{N}
\bigl(
a_n\cos(n\theta)+b_n\sin(n\theta)
\bigr).
\]
Then
\[
\frac{u(r,\theta)}
{e^{-x_3/2}K_0(r/2)}
\longrightarrow A(\theta)
\qquad\text{as }r\to\infty,
\]
and \eqref{eq:cos-mode-flux}--\eqref{eq:sin-mode-flux} give
\begin{equation}
q_{\mathrm{lin}}(u)
=
-2\pi A\left(\frac{\pi}{2}\right).
\label{eq:linear-flux-angular-profile}
\end{equation}
Thus, for finite sums of modes, the weighted horizontal flux
detects the value of the leading angular profile $A(\theta)$ in the upward
$\ee_3$-direction.

This linear calculation also shows why, for a
general bounded exterior solution, an asymptotic expansion as a simple multiple of $G_L$ should not be expected: already at the linearized level, infinitely
many angular modes occur with the same order of decay as the Green's function.

\section{Vertical flux of planar ends}\label{sec:vert_flux}
In this section we record some consequences of the sharp decay estimates
for bounded exterior solutions obtained in Section~2 in the case of vertical flux. Besides Theorem \ref{thm:planar_wings}, we are going to see that the vertical flux encodes important information about the behavior at the boundary of a planar end. This information may potentially be useful for matching data in gluing techniques. 

\subsection{Nodal sets of the first Gauss map component}
Let $\Sigma\subset\mathbb{R}^3$ be a translator and set
\begin{equation}
\psi=\langle \nu,\ee_1\rangle,
\end{equation}
where $\nu$ is a unit normal to $\Sigma$.
Then by direct calculation, or geometrically using that motion in the direction of $\ee_1$ produces a Killing field, we have the following well-known fact:
\begin{equation}
    \left(\Delta + \ee_3\cdot\nabla + |A|^2\right)\psi = 0,
\end{equation}
so that $\psi$ is harmonic with respect to the Jacobi operator $\Delta + \ee_3\cdot\nabla + |A|^2$. In this sense, the zero set of $\psi$ is a nodal set.

Suppose that $E$ is a connected component of $\Sigma\setminus \{\langle\nu,\ee_1\rangle=0\}.$ WLOG assume that
\[
\psi>0
\qquad\text{on }  E.
\]
Since
$
\partial E
\subset
\{\psi=0\},
$
one has
\[
\ee_1\in T_p\Sigma\qquad \forall p\in\partial E.
\]

Suppose that $\Gamma$ is a regular component of
$\{\psi=0\}$ contained in $\partial E$. Then by the Hopf Lemma
\[
\nabla_\Sigma\psi\neq0
\qquad\text{along }\Gamma,
\]
and the unit conormal to $\Gamma$ in $\Sigma$ is necessarily parallel
to $\nabla_\Sigma\psi$. More precisely, since
$E\subseteq\{\psi>0\}$, the outward conormal is
\begin{equation}
\eta
=
-\frac{\nabla_\Sigma\psi}
{|\nabla_\Sigma\psi|}.
\label{eq:conormal-angle}
\end{equation}

On the other hand, differentiating
\[
\psi=\langle\nu,\ee_1\rangle
\]
along $\Sigma$ gives
\begin{equation}
\nabla_\Sigma\psi
=
-A(\ee_1^\top),
\label{eq:gradient-angle}
\end{equation}
where $A$ denotes the shape operator. Along $\Gamma$ we have
$\psi=0$, and hence $\ee_1^\top=\ee_1$. Thus
\begin{equation}
\nabla_\Sigma\psi
=
-A(\ee_1)
\qquad\text{on }\Gamma.
\label{eq:gradient-angle-boundary}
\end{equation}
Consequently,
\begin{equation}
\eta
=
\frac{A(\ee_1)}{|A(\ee_1)|}
\qquad\text{on the regular part of }\Gamma,
\label{eq:conormal-second-fundamental}
\end{equation}
with the above choice of orientation.

In particular, the condition
\[
\eta=\pm\ee_1
\]
holds true if, and only if,  $\ee_1$ is a
principal direction along $\Gamma$.

\subsection{The vertical flux}

Suppose now $E\subseteq\Sigma\setminus \{\langle\nu,\ee_1\rangle=0\}$ is a horizontal collapsed graphical end in the direction of $\ee_1$, in the sense of Definition \ref{dfn:collapsed_graphical_end}. Moreover, assume that the finite boundary of
$E$ consists of a regular closed curve
\(
\partial E
=
\Gamma.
\)
Let $\eta$ denote the conormal pointing outward from $E$ along this
finite boundary.

For $R$ sufficiently large, truncate $E$ by the cylinder
\[
C_R
=
\left\{
x_2^2+x_3^2\leq R^2
\right\},
\]
and denote
\[
E(R)=E\cap C_R.
\]
The Neves--Tian identity (Lemma \ref{lem:neves_tian}) gives
\begin{equation}
\int_{E(R)}|\vec H|^2\,dA
=
\int_{\Gamma}
\langle\eta,\ee_3\rangle\,ds
+
\int_{\partial E(R)\cap\partial C_R}
\langle\eta_R,\ee_3\rangle\,ds.
\label{eq:NT-maximal-graph}
\end{equation}
Since the outer part of $E$ agrees with the planar graphical end, the
decay estimate for the vertical flux gives
\[
\int_{\partial E(R)\cap\partial C_R}
\langle\eta_R,\ee_3\rangle\,ds
\longrightarrow0
\qquad\text{as }R\to\infty.
\]
Therefore,
\begin{equation}
\int_{E}|\vec H|^2\,dA
=
\int_{\Gamma}
\langle\eta,\ee_3\rangle\,ds.
\label{eq:vertical-flux-maximal-graph}
\end{equation}

Thus the $\ee_3$-flux through the finite boundary of $E$ measures precisely the total $L^2$-norm of the mean curvature contained in that
region.

Using~\eqref{eq:conormal-angle}, this can also be expressed directly in
terms of the angle function:
\begin{equation}
\langle\eta,\ee_3\rangle
=
-
\frac{
\langle\nabla_\Sigma\psi,\ee_3^\top\rangle
}{
|\nabla_\Sigma\psi|
}.
\label{eq:vertical-flux-angle}
\end{equation}
Equivalently, by~\eqref{eq:gradient-angle-boundary},
\begin{equation}
\langle\eta,\ee_3\rangle
=
\frac{
\langle A(\ee_1),\ee_3^\top\rangle
}{
|A(\ee_1)|
}
\qquad\text{along }\Gamma.
\label{eq:vertical-flux-second-fundamental}
\end{equation}

Hence the vertical flux through the neck region is related to the way
in which the horizontal angle function
\[
\psi=\langle\nu,\ee_1\rangle
\]
crosses zero at the boundary of $E$.

\subsection{A rigidity observation}

There is also a useful consequence of this identity. Suppose, for
instance, that the conormal along the whole finite boundary were equal
to $\pm \ee_1$. Since
\[
\langle \ee_1,\ee_3\rangle=0,
\]
we would obtain
\[
\int_{\Gamma}
\langle\eta,\ee_3\rangle\,ds=0.
\]
Formula~\eqref{eq:vertical-flux-maximal-graph} would then
imply
\[
\int_{E}|\vec H|^2\,dA=0.
\]
Consequently,
\[
\vec H\equiv0
\qquad\text{on }E.
\]
Since $\Sigma$ is a translator,
\[
\vec H=\ee_3^\perp,
\]
and therefore $\ee_3$ is tangent to $E$ everywhere. It follows that
$E$ is a vertical minimal surface, hence a portion of a vertical
plane.

Thus, unless the translator is flat, the finite boundary of $E$ cannot have conormal identically equal to
$\pm \ee_1$.

This observation suggests that the vertical flux contains genuinely
nontrivial information about the geometry of the interface between a
planar end and the neck region. In particular, understanding the sign
of
\[
\langle\eta,\ee_3\rangle
=
-
\frac{
\langle\nabla_\Sigma\psi,\ee_3^\top\rangle
}{
|\nabla_\Sigma\psi|
}
\]
along the regular part of $\partial E$ could provide useful
information about both the geometry of the neck and the curvature
energy carried by $E$.

As in \cite{HMW-1}, note that reflection in a vertical plane is an isometry of the Ilmanen metric and preserves translators. If 
$$\eta=\pm \ee_1,$$
along a boundary component $\Gamma$ then $x_1$  is constant along $\Gamma$, so $\Gamma$ lies in a vertical plane 
$$V=\{x_1=c\},$$ and the surface meets $V$ orthogonally. Under the appropriate regularity assumptions one can therefore reflect across $V$. The doubled surface has only planar ends; if it is complete, proper, and of finite topology (see also Remark \ref{rem:Neves-Tian}), Theorem \ref{thm:planar_wings} would immediately imply that the double is a vertical plane. This gives a second route to essentially the same nonexistence result. So, this rigidity result can also be reformulated as follows:
 \begin{thm}
     A collapsed graphical end of a translator, in the sense of Definition \ref{dfn:collapsed_graphical_end}, cannot meet a vertical plane orthogonally along its entire compact boundary.
 \end{thm}

This formulation also elucidates the relation between free-boundary translators and translators without boundary obtained by reflection.

The preceding flux argument is not specific to translating solitons.
Indeed, consider more generally a $[\varphi,\ee_3]$-minimal surface
$\Sigma\subset\mathbb{R}^3$, that is,
\[
\vec H=\dot\varphi(x_3)\ee_3^\perp.
\]
These surfaces are critical points of the weighted area functional
$\int_\Sigma e^\varphi\,dA$; see
Martínez--Martínez-Triviño~\cite{Trivino}.
If $\mu=x_3|_\Sigma$ denotes the height function, then
\cite[Lemma~2.1]{Trivino} gives
\[
\Delta_\Sigma\mu
=\dot\varphi(\mu)\bigl(1-|\nabla_\Sigma\mu|^2\bigr),
\qquad
\dot\varphi(\mu)^2
=\dot\varphi(\mu)^2|\nabla_\Sigma\mu|^2+|\vec H|^2.
\]
Consequently, wherever $\dot\varphi\neq0$,
\[
\Delta_\Sigma\mu
=
\frac{|\vec H|^2}{\dot\varphi(\mu)},
\]
and hence, for every compact domain $M\subset\Sigma$,
\[
\int_{\partial M}\langle\eta,\ee_3\rangle\,ds
=
\int_M\frac{|\vec H|^2}{\dot\varphi(x_3)}\,dA.
\]
Thus, whenever $\dot\varphi$ has a fixed nonzero sign, the same
rigidity mechanism applies: if the vertical flux vanishes both at
infinity and along the finite boundary of a graphical region, then
$H\equiv0$ there. The translator case corresponds to
$\varphi(x_3)=x_3$ and recovers the flux identities used above, see also Neves-Tian~\cite{Neves-Tian}.

\subsection{Vertical flux and weighted area excess}
\label{subsec:vertical-flux}

The flux in the translating direction has a different nature from the
horizontal flux considered above. Rather than being conserved, the
ordinary $\ee_3$-flux measures exactly the curvature energy contained in
the remaining part of a planar end. A weighted version admits a similar
interpretation in terms of the weighted area excess over the asymptotic
vertical plane.

Let
\[
X(x_2,x_3)=\bigl(u(x_2,x_3),x_2,x_3\bigr),
\qquad (x_2,x_3)\in A_{R_0},
\]
be a horizontal collapsed graphical end as in Definition \ref{dfn:collapsed_graphical_end}, and let
\[
E_R:=X(A_R),
\qquad
A_R=\mathbb{R}^2\setminus B_R(0).
\]
For $R>R_0$, denote by $\eta_R$ the conormal to $X(S_R)$ pointing
toward infinity in $E_R$, and define
\begin{equation}
\mathcal{F}_3(R)
:=
\int_{X(S_R)}
\langle \eta_R,\ee_3\rangle\,ds_\Sigma .
\label{eq:vertical-flux}
\end{equation}

We first observe that $\mathcal{F}_3(R)$ has an exact geometric
interpretation. Let $R<\rho$ and consider the compact annular region
\[
E_{R,\rho}:=X\bigl(B_\rho(0)\setminus B_R(0)\bigr).
\]
By the Neves--Tian flux identity (Lemma \ref{lem:neves_tian}),
\[
\int_{E_{R,\rho}} |\vec H|^2\,dA
=
\int_{\partial E_{R,\rho}}
\langle\eta,\ee_3\rangle\,ds_\Sigma .
\]
With our choice of orientations, this becomes
\begin{equation}
\int_{E_{R,\rho}} |\vec H|^2\,dA
=
\mathcal{F}_3(\rho)-\mathcal{F}_3(R).
\label{eq:vertical-flux-annulus}
\end{equation}
On the other hand, Lemma~\ref{lem:graph_flux_decay} gives
\[
\mathcal{F}_3(\rho)\longrightarrow 0
\qquad\text{as }\rho\to\infty.
\]
Letting $\rho\to\infty$ in
\eqref{eq:vertical-flux-annulus}, we obtain
\begin{equation}
\mathcal{F}_3(R)
=
-\int_{E_R}|\vec H|^2\,dA .
\label{eq:vertical-flux-energy}
\end{equation}
In particular,
\[
\mathcal{F}_3(R)\leq 0
\]
and
\[
\mathcal{F}_3(R)\nearrow 0
\qquad\text{as }R\to\infty.
\]
Thus the vertical flux through a cross-section of a planar end measures
exactly the curvature energy remaining beyond that cross-section. See also Section \ref{sec:Poho}.

There is a closely related weighted identity. Since $\Sigma$ is a
translator,
\[
\vec H=\ee_3^\perp,
\]
and hence
\[
|\ee_3^\top|^2+|\vec H|^2=1.
\]
Moreover,
\[
\operatorname{div}_\Sigma(\ee_3^\top)
=
\Delta_\Sigma x_3
=
|\vec H|^2.
\]
It follows that
\begin{align}
\operatorname{div}_\Sigma
\left(e^{x_3}\ee_3^\top\right)
&=
e^{x_3}
\left(
|\ee_3^\top|^2+
\operatorname{div}_\Sigma(\ee_3^\top)
\right)
\nonumber\\
&=
e^{x_3}.
\label{eq:weighted-vertical-divergence}
\end{align}

The corresponding weighted vertical flux itself is not expected to
decay on a planar end, since the asymptotic vertical plane contributes
a nonzero term. It is therefore natural to subtract the flux of that
plane. Define
\begin{equation}
\widetilde{\mathcal{F}}_3(R)
:=
\int_{X(S_R)}
e^{x_3}\langle\eta_R,\ee_3\rangle\,ds_\Sigma
-
\int_{S_R}e^{x_3}\mu_3\,ds_0,
\label{eq:renormalized-vertical-flux}
\end{equation}
where $\mu=(\mu_2,\mu_3)$ is the outward Euclidean unit normal to
$S_R$.

Recall that, if
$
W=\sqrt{1+|\nabla u|^2},
$
then
\[
\langle\eta_R,\ee_3\rangle\,ds_\Sigma
=
\left(
W\mu_3-\frac{u_\mu u_3}{W}
\right)ds_0.
\]
Consequently,
\begin{equation}
\widetilde{\mathcal{F}}_3(R)
=
\int_{S_R}
e^{x_3}
\left[
(W-1)\mu_3
-
\frac{u_\mu u_3}{W}
\right]ds_0.
\label{eq:renormalized-flux-graph}
\end{equation}
Using (as in \eqref{eq:W_diff})
\[
|W-1|\leq C|\nabla u|^2
\]
and
\[
\left|\frac{u_\mu u_3}{W}\right|
\leq |\nabla u|^2,
\]
we get
\[
|\widetilde{\mathcal{F}}_3(R)|
\leq
C\int_{S_R}e^{x_3}|\nabla u|^2\,ds_0.
\]
The sharp gradient estimate
\[
|\nabla u(x)|
\leq
C e^{-x_3/2}
K_0\left(\frac{|x|}{2}\right)
\]
therefore gives
\[
|\widetilde{\mathcal{F}}_3(R)|
\leq
CRK_0\left(\frac{R}{2}\right)^2
=
O(e^{-R}).
\]
In particular,
\begin{equation}
\widetilde{\mathcal{F}}_3(R)
\longrightarrow0
\qquad\text{as }R\to\infty.
\label{eq:renormalized-flux-zero}
\end{equation}

We can now identify this renormalized flux exactly. Applying
\eqref{eq:weighted-vertical-divergence} to
$E_{R,\rho}$ gives
\[
\int_{X(S_\rho)}
e^{x_3}\langle\eta,\ee_3\rangle\,ds_\Sigma
-
\int_{X(S_R)}
e^{x_3}\langle\eta,\ee_3\rangle\,ds_\Sigma
=
\int_{B_\rho\setminus B_R}e^{x_3}W\,dx_2dx_3.
\]
Applying the same identity to the vertical plane gives
\[
\int_{S_\rho}e^{x_3}\mu_3\,ds_0
-
\int_{S_R}e^{x_3}\mu_3\,ds_0
=
\int_{B_\rho\setminus B_R}e^{x_3}\,dx_2dx_3.
\]
Subtracting these two identities yields
\[
\widetilde{\mathcal{F}}_3(\rho)
-
\widetilde{\mathcal{F}}_3(R)
=
\int_{B_\rho\setminus B_R}
e^{x_3}(W-1)\,dx_2dx_3.
\]
Letting $\rho\to\infty$ and using
\eqref{eq:renormalized-flux-zero}, we arrive at
\begin{equation}
-\widetilde{\mathcal{F}}_3(R)
=
\int_{A_R}
e^{x_3}
\left(
\sqrt{1+|\nabla u|^2}-1
\right)
\,dx_2dx_3.
\label{eq:weighted-area-excess}
\end{equation}

Thus the negative of the renormalized vertical flux is precisely the
weighted area excess of the planar end over its asymptotic vertical
plane. In particular, this quantity is finite and nonnegative.

The identities
\eqref{eq:vertical-flux-energy} and
\eqref{eq:weighted-area-excess} provide two complementary measurements
of the geometry of a planar end. The ordinary vertical flux measures
the remaining $L^2$-curvature,
\[
-\mathcal{F}_3(R)
=
\int_{E_R}|\vec H|^2\,dA,
\]
whereas the renormalized weighted vertical flux measures the remaining
weighted area excess,
\[
-\widetilde{\mathcal{F}}_3(R)
=
\int_{A_R}
e^{x_3}(W-1)\,dx_2dx_3.
\]
Together with the conserved horizontal flux, these identities provide
natural quantities associated with a planar exterior piece and may be
useful as matching data in gluing constructions.


\section{Pohozaev-type interpretations}\label{sec:Poho}

There is a natural way to place the preceding flux argument in the framework of Pohozaev-type identities. Let $X$ be a conformal Killing field
of $\mathbb{R}^3$, so that
\[
\mathcal{L}_X g_{\mathrm{Euc}}=2\varphi\,g_{\mathrm{Euc}}
\]
for some function $\varphi$. On any surface holds
\[
\operatorname{div}_\Sigma X^\top
=
2\varphi+\langle X,\vec H\rangle.
\]
If $M\subset\Sigma$ is a compact domain in a
translator, since
\[
\vec H=\ee_3^\perp,
\]
integration gives the Pohozaev--Minkowski identity
\begin{equation}
\int_{\partial M}\langle X,\eta\rangle\,ds
=
\int_M
\left(
2\varphi+\langle X,\ee_3^\perp\rangle
\right)dA,
\label{eq:pohozaev-translator}
\end{equation}
where $\eta$ is the outward unit conormal.

Integral formulae of this type for hypersurfaces, or more generally curvature problems in Riemannian geometry, in the presence of conformal Killing fields are classical; see, for instance, \cite{AlbujerAledoAlias, BentRod}. Related Pohozaev-type obstructions involving
conformal vector fields appear in \cite{AmmannHumbertAhmedou}.

The vertical flux identity used above is the simplest member of
\eqref{eq:pohozaev-translator}. Namely, taking $X=\ee_3$, which is a
Killing field of $(\mathbb{R}^3,g_{\mathrm{Euc}})$, which is a conformal Killing field for the Ilmanen metric, gives $\varphi=0$ and
\[
\langle \ee_3,\ee_3^\perp\rangle
=
|\ee_3^\perp|^2
=
|\vec H|^2.
\]
Consequently, we recover \eqref{NTas}:
\begin{equation}\label{PohoNTas}
\int_{\partial M}\langle\eta,\ee_3\rangle\,ds
=
\int_M|\vec H|^2\,dA.
\end{equation}
Thus the nonexistence argument considered above may itself be viewed as
a Pohozaev-type argument: if the finite boundary of a planar-ended
region has conormal $\eta=\pm \ee_1$, then its contribution to the
$\ee_3$-flux vanishes, while the contribution at the planar end tends to
zero. The right-hand side is nonnegative, and hence
$\vec H\equiv0$.

Proper conformal Killing fields give additional identities. For
example, for the dilation field
\[
X(x)=x-p
\]
one has $\varphi=1$, and therefore
\begin{equation}
\int_{\partial M}
\langle x-p,\eta\rangle\,ds
=
\int_M
\left(
2+\langle x-p,\ee_3^\perp\rangle
\right)dA.
\label{eq:dilation-pohozaev}
\end{equation}
Such identities may possibly provide further obstructions when the geometry of
the finite boundary makes the left-hand side vanish or have a
prescribed sign. In the noncompact planar-ended setting the boundary
term at infinity need not vanish for a dilation field, so one should
expect to use a renormalized version of
\eqref{eq:dilation-pohozaev}, subtracting the corresponding identity
for the asymptotic vertical plane. The sharp decay estimates proved in
this paper, namely in Theorem \ref{thm:optimal_decay}, are strong enough to make such renormalized boundary terms
natural candidates for finite quantities.

There is also a particularly simple weighted formulation. Recall that
a translator is minimal for the Ilmanen metric
\[
g_I=e^{x_3}g_{\mathrm{Euc}}.
\]
If $X$ is Euclidean conformal Killing with conformal factor $\varphi$, then
\[
\mathcal L_X g_I
=
\bigl(2\varphi+\langle X,\ee_3\rangle\bigr)g_I.
\]
Since $\Sigma$ is minimal with respect to $g_I$, the corresponding
Minkowski formula becomes
\begin{equation}
\int_{\partial M}
e^{x_3}\langle X,\eta\rangle\,ds
=
\int_M
e^{x_3}
\bigl(2\varphi+\langle X,\ee_3\rangle\bigr)\,dA.
\label{eq:weighted-pohozaev}
\end{equation}
This single identity contains several of the flux formulae considered
above. Horizontal translations $X=\ee_1,\ee_2$ give the conserved weighted
horizontal fluxes, whereas $X=\ee_3$ gives
\[
\int_{\partial M}
e^{x_3}\langle\eta,\ee_3\rangle\,ds
=
\int_M e^{x_3}\,dA.
\]
For the dilation field $X=x-p$, one obtains
\[
\int_{\partial M}
e^{x_3}\langle x-p,\eta\rangle\,ds
=
\int_M
e^{x_3}\bigl(2+x_3-p_3\bigr)\,dA.
\]
It would in this connection be interesting to investigate further whether suitable choices of conformal Killing fields yield further nonexistence or rigidity statements for planar-ended translators.


\end{document}